\documentclass[11pt, twoside]{article}
\usepackage{bcc-cambridge-class}

\usepackage{tikz}
\usepackage{float}
\usepackage{mathtools}

\let\a\relax
\let\b\relax
\let\dim\relax
\let\div\relax
\let\N\relax
\let\span\relax
\let\t\relax
\let\v\relax
\let\S\relax

\DeclareMathOperator{\a}{{\mathbf a}}  
\DeclareMathOperator{\b}{{\mathbf b}}  
\DeclareMathOperator{\e}{{\mathbf e}}  
\DeclareMathOperator{\m}{{\mathbf m}}  
\DeclareMathOperator{\t}{{\mathbf t}}  
\DeclareMathOperator{\v}{{\mathbf v}}  
\DeclareMathOperator{\x}{{\mathbf x}}  
\DeclareMathOperator{\y}{{\mathbf y}}  

\DeclareMathOperator{\M}{{\mathsf M}}

\DeclareMathOperator{\F}{{\mathcal F}}  
\DeclareMathOperator{\G}{{\mathcal G}}  
  
\DeclareMathOperator{\S}{{\mathcal S}}

\DeclareMathOperator{\CC}{{\mathbb C}}  
\DeclareMathOperator{\FF}{{\mathbb F}}  

\DeclareMathOperator{\RR}{{\mathbb R}}  
\DeclareMathOperator{\ZZ}{{\mathbb Z}}

\DeclareMathOperator{\cone}{{\mathrm cone}}  
\DeclareMathOperator{\conv}{{\mathrm conv}}  
\DeclareMathOperator{\csm}{{\mathrm csm}}  

\DeclareMathOperator{\dim}{{\mathrm dim}}  
\DeclareMathOperator{\div}{{\mathrm div}}  
\DeclareMathOperator{\mult}{{\mathrm mult}}
\DeclareMathOperator{\span}{{\mathrm span}}  
\DeclareMathOperator{\trop}{{\mathrm trop}}  

\DeclareMathOperator{\Hom}{{\mathrm Hom}}  
\DeclareMathOperator{\MW}{{\mathrm MW}}  
\DeclareMathOperator{\N}{\mathrm N}   
\DeclareMathOperator{\Newt}{{\mathrm Newt}}  
\DeclareMathOperator{\PP}{{\mathrm PP}}  
\DeclareMathOperator{\Sym}{\mathrm Sym}   
\DeclareMathOperator{\Tr}{{\mathrm Tr}}

\bcctitle{Intersection theory of matroids: \\  variations on a theme}
\bccshorttitle{Intersection theory of matroids}

\bccname{Federico Ardila--Mantilla}
\bccshortname{Federico Ardila--Mantilla} 

\bccaddressa{San Francisco State University \\ 1600 Holloway Avenue \\ San Francisco, CA 94132 \\ United States}
\bccemaila{federico@sfsu.edu}

\begin{document}

\makebcctitle


\begin{abstract}
Chow rings of toric varieties, which originate in intersection theory, feature a rich combinatorial structure of independent interest. 
We survey four different ways of computing in these rings, due to 
Billera, 
Brion, 
Fulton--Sturmfels, and
Allermann--Rau.
We illustrate the beauty and power of these methods by giving four proofs of Huh and Huh--Katz's formula 
$\mu^k(\M) = \deg_{\M}(\alpha^{r-k} \beta^k)$
for the coefficients of the reduced characteristic polynomial of a matroid $\M$ as the mixed intersection numbers of the hyperplane and reciprocal hyperplane classes $\alpha$ and $\beta$ in the Chow ring of $\M$. Each of these proofs sheds light on a different aspect of matroid combinatorics, and provides a framework for further developments in the intersection theory of matroids.

Our presentation is combinatorial, and does not assume previous knowledge of toric varieties, Chow rings, or intersection theory. 
This survey was prepared for the Clay Lecture to be delivered at the 2024 British Combinatorics Conference.
\end{abstract}


%

\section{Introduction}\label{sec:intro}

Our starting point is the \defword{chromatic polynomial} $\chi_G(t)$ of a graph $G=(V,E)$. 
For a positive integer $q$,
\[
\chi_G(q) \coloneqq \text{ number of proper vertex-colorings of G with $q$ colors},
\]
where a coloring is \defword{proper} if no two neighboring vertices have the same color. 
For example, the chromatic polynomial of the graph below is $\chi_G(q) = q(q-1)^2(q-2).$

\begin{figure}[h]
 \begin{center}
    \includegraphics[height=2cm]{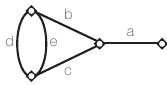}
  \caption{  \label{fig:graph} A graph $G$ with $\chi_G(q) = q^4-4q^3 +5q^2- 2q.$ and  $\mu^0=1, \mu^1=3, \mu^2=2$.}
 \end{center}
\end{figure}

More generally, the \defword{characteristic polynomial} $\chi_{\M}(t)$ of a matroid $\M=(E,r)$ is
\begin{equation} \label{eq:whitney} \chi_{\M}(q) \coloneqq \sum_{A \subseteq E} (-1)^{|A|} q^{r - r(A)}. 
\end{equation} It is one of the most important invariants of a matroid; it is introduced in detail in Section \ref{sec:matroid} and \cite[Sections 6, 7]{yomethods}.
The characteristic polynomial generalizes the chromatic polynomial in the sense that if $\M(G)$ is the cycle matroid of a graph $G$ that has $c$ connected components, then $\chi_G(q) = q^c\chi_{\M(G)}(q)$. 
This polynomial is a multiple of $q-1$, and we define the \defword{reduced characteristic polynomial} of $\M$ to be
\[
\overline{\chi}_{\M}(q) \coloneqq \frac{\chi_{\M}(q)}{q-1} = \mu^0 q^r - \mu^1 q^{r-1} + \cdots + (-1)^r \mu^r q^0
\]
where $r+1$ is the rank of $\M$. In the example above, $\overline{\chi}_{\M}(q) = q^2-3q+2$. 

It is not too difficult to prove recursively that the numbers $\mu^0, \mu^1, \ldots, \mu^r$ are non-negative. A combinatorialist then asks: Do they count something? An algebraic combinatorialist then asks: Do they have an algebraic, geometric, or topological interpretation? Such questions often give rise to a deeper understanding of the objects under study. In this case, and in numerous others, they lead to proofs of long-standing conjectures for which no purely combinatorial proof is known.

\subsection{Theme}

Our recurring theme will be Huh \cite{Huh} and Huh-Katz \cite{HuhKatz}'s remarkable interpretation of $\mu^0, \ldots, \mu^r$.  
Beautiful in its own right, their Theorem \ref{thm:main} also lies at the heart of the celebrated proof of the conjecture that this sequence is log-concave \cite{AdiprasitoHuhKatz}.

Let $\M$ be a matroid  of rank $r+1$ on a set $E$ with $n+1$ elements.
The \defword{Chow ring}
$A(\M)$ is the $\RR$-algebra generated by variables $x_F$ for each non-empty proper flat, with relations
\begin{eqnarray*}
x_Fx_G=0 && \text{ for any flats $F, G$ such that $F \subsetneq G$ and $F \supsetneq G$}, \\
\sum_{F \ni i}x_F = \sum_{F \ni j} x_F && \text{ for any elements $i, j \in E$}.
\end{eqnarray*}
One can show that the Chow ring is graded $A(\M) = A^0(\M) \oplus \cdots \oplus A^r(\M)$, and that there is a canonical isomorphism $\deg_{\M}: A^r(\M) \xrightarrow{\sim} \mathbb{R}$ called the \defword{degree map} \cite{AdiprasitoHuhKatz}. 

Consider the following two elements of $A^1(\M)$, which we call the \defword{hyperplane} and \defword{reciprocal hyperplane} classes:
\[
\alpha = \alpha_i =  \sum_{i \in F} x_F, \qquad
\beta = \beta_i =  \sum_{i \notin F} x_F.
\]
One readily verifies that they do not depend on $i$. 

\bigskip

\framebox{
\begin{minipage}{12.5cm}

\begin{theorem}\label{thm:main}
Let $\M$ be a matroid  of rank $r+1$. Let $\alpha, \beta$ be the  hyperplane and reciprocal hyperplane classes in the Chow ring $A(\M)$. Then
\[
\deg_{\M}(\alpha^{r-k} \beta^k) = \mu^k(\M) \qquad \text{ for } 0 \leq k \leq r.
\]
\end{theorem}

\end{minipage}}

\bigskip

The Chow ring $A(\M)$ has remarkable Hodge-theoretic properties \cite{AdiprasitoHuhKatz} surveyed in \cite{HuhICM, yoICM, BakeronHuh, EuronHuh}.
In particular, $A(\M)$ satisfies the \defword{Hodge--Riemann relations}, which give
\[
\deg_{\M}(\ell_1 \ell_2 \ell_3 \cdots \ell_d)^2 \ge \deg_{\M}(\ell_1 \ell_1 \ell_3 \cdots \ell_d) \deg_{\M}(\ell_2 \ell_2 \ell_3 \cdots \ell_d),
\]
for any $\ell_1,\ell_2,\ldots,\ell_d$ in a certain cone $\mathcal{K}(\M) \subseteq A^1(\M)$ whose closure contains $\alpha$ and $\beta$. In light of Theorem \ref{thm:main}, this proves the following inequalities conjectured by  Rota \cite{Rota}, Heron \cite{Heron}, and Welsh \cite{Welsh} in the 1970s:
\[
(\mu^k)^2 \geq \mu^{k+1}\mu^{k-1} \qquad \text{ for } 1 \leq k \leq r-1. 
\]

This survey focuses on the combinatorial aspects of this program:

\begin{qn}
How does one discover and prove combinatorially interesting formulas in Chow rings like Theorem \ref{thm:main}? 
\end{qn}

This question fits within the framework of intersection theory of toric varieties, in ways that can be understood combinatorially. 
The Chow ring $A(X_\Sigma)$ of a toric variety $X_{\Sigma}$ corresponding to a rational polyhedral fan $\Sigma$ is a beautifully rich object that can be understood from several different points of view. We will present four, due to 
Billera, 
Brion, 
Fulton--Sturmfels, and
Allermann--Rau
 \cite{Billera, Brion, FultonSturmfels, AllermannRau}. Each of these points of view gives us a different ways to compute in a Chow ring, and teaches us different things about the objects at hand. This machinery is relevant to Theorem \ref{thm:main} because the Chow ring of a matroid $\M$ equals the Chow ring of the toric variety $X_{\Sigma_{\M}}$ and is closely related to the permutahedral toric variety $X_{\Sigma_E}$,
 where $\Sigma_E$ and  $\Sigma_{\M}$ are the \emph{matroid fan} of $\M$ 
 and the  \emph{braid fan} of $E$, discussed in detail in Sections \ref{subsec:braidfan} and \ref{subsec:matroid}. 
 
Our presentation will be combinatorial, and will not assume previous knowledge of toric varieties, Chow rings, or intersection theory. A familiarity with the basics of enumerative matroid theory will be helpful; see for example \cite{yomethods, Bjorner, Oxley}.

This survey is organized as follows. In Section \ref{sec:toric} we discuss the general intersection theory of simplicial rational fans $\Sigma$ and toric varieties $X_\Sigma$, giving four different combinatorial points of view on the Chow ring $A(\Sigma) = A(X_\Sigma)$. We pay special attention to the Chow ring of the braid fan $\Sigma_E$ for a finite set $E$.
In Section \ref{sec:matroid} we discuss some basic aspects of the intersection theory of matroids. The general theory gives us four different ways to think about the Chow ring of a matroid $\M$. We illustrate each one of these approaches by using it to give a different proof of Theorem \ref{thm:main}. 

\section{Intersection theory of toric varieties: a case study}\label{sec:toric}

Intersection theory studies how subvarieties of an algebraic variety $X$ intersect. For example, Bezout's theorem tells us that two generic plane curves of degrees $m$ and $n$ intersect at $mn$ points. We want a robust theory that will keep track of multiplicities correctly, and where the answer to such intersection questions does not change under rational equivalence. The Chow ring $A(X)$ provides an algebraic framework to carry out such computations. 
Because this ring encodes the answers to very subtle questions, it is generally an unwieldy object.

The situation is much better behaved when $X=X_\Sigma$ is the toric variety associated to a simplicial rational fan $\Sigma$. In this case, the Chow ring $A(X_\Sigma)$ can be described entirely in terms of the fan $\Sigma$ in several ways. This leads to algebraic, geometric, and combinatorial methods for computing in $A(X_\Sigma)$, and to combinatorial results of independent interest. Those methods and results are the subjects of this survey.

Let $\N_{\mathbb{Z}} \cong \ZZ^n$ be a lattice and $\N=\mathbb{R} \otimes \N_\mathbb{Z} \cong \RR^n$ the corresponding real vector space.
A \defword{rational cone} $\{\lambda_1 \v_1 + \cdots + \lambda_k \v_k \, : \, \lambda_1, \ldots, \lambda_k \geq 0\}$ is a cone in $\N$  generated by finitely many lattice vectors $\v_1, \ldots, \v_k \in \N_{\ZZ}$; it is \defword{strongly convex} if it contains no lines.
A \defword{rational fan} $\Sigma$ in $\N$ is a set of strongly convex rational cones that are glued along common faces; that is, any face of a cone in $\Sigma$ is a cone in $\Sigma$, and the intersection of any two cones in $\Sigma$ is a cone in $\Sigma$. We say a fan $\Sigma$ is \defword{simplicial} if every $d$-dimensional cone is generated by $d$ vectors, \defword{unimodular} if those $d$ vectors always form a basis for $\N_{\mathbb{Z}}$, 
 and \defword{complete} if the union of the cones in $\Sigma$ is all of $\N$. We say $\Sigma$ is \emph{pure} if all maximal cones have the same dimension, and write $\Sigma(d)$ for the set of $d$-dimensional cones.
 A rational fan $\Sigma$ in $\N$ determines a toric variety $X=X_{\Sigma}$; for details see \cite{CoxLittleSchenck, Fulton}.

The goal of this section is to explain the following theorem. After explaining each of its parts, we use it to compute explicitly the Chow ring of the two-dimensional braid fan. 

\begin{theorem}
Let $\Sigma$ be a complete simplicial rational fan in  $\N=\mathbb{R} \otimes \N_\mathbb{Z}$. 
The following rings are isomorphic:

\begin{enumerate}
\item The quotient $A(\Sigma) = S(\Sigma)/(I(\Sigma)+J(\Sigma))$ where 
\begin{eqnarray*}
S(\Sigma) &=& \RR[x_\rho \, : \, \rho \text{ is a ray of } \Sigma]/(I(\Sigma)+J(\Sigma)), \\
I(\Sigma) &=& \langle{ x_{\rho_1} \cdots x_{\rho_k} \, : \, \rho_1, \ldots, \rho_k \text{ do not generate a cone of } \Sigma \rangle},  \\
J(\Sigma) &=& \langle{\sum_{\rho \text{ ray of } \Sigma} \ell(\e_\rho) x_\rho \, : \, \ell  \text{ is a linear function on } \N \rangle}.
\end{eqnarray*}

\item The ring $\PP(\Sigma)/\langle \N^\vee \rangle$ of piecewise polynomial functions on $\Sigma$ modulo the ideal generated by the space  $N^\vee$  of (global) linear functions on $N$.

\item The ring $\MW(\Sigma)$ of Minkowski weights on $\Sigma$ under stable intersection.

\item The ring $\MW(\Sigma)$ of Minkowski weights on $\Sigma$ under tropical intersection.

\item The cohomology ring of the toric variety $X(\Sigma)$. 

\item The Chow ring of the toric variety $X(\Sigma)$.
\end{enumerate}
When $\Sigma$ is unimodular, there are analogous isomorphisms over $\ZZ$.
\end{theorem}

\setcounter{subsection}{-1}

\subsection{The braid fan}\label{subsec:braidfan}

For a finite set $E$, we let $\{\e_i \, : \, i \in E\}$ be the standard basis of $\RR^E$, and we write 
\[
\e_S \coloneqq \sum_{s \in S} \e_s \qquad \text{ for } S \subseteq E.
\]
The fans considered in this paper
will live in $\N_E \coloneqq \RR^E/\RR\e_E$. The image of $\e_S \in \RR^E$ in this quotient will also be denoted $\e_S \in \N_E$. We will often consider $E = [0,n]\coloneqq \{0, 1, \ldots, n\}$.

\begin{Def} Let $E$ be a finite set. 
The \defword{braid fan} $\Sigma_E$ in $\N_E \coloneqq \RR^E/\RR\e_E$ has
\begin{itemize}
\item rays: $\e_S$ for the nonempty proper subsets $\emptyset \subsetneq S \subsetneq E$
\item cones: $\sigma_{\S} = \cone(\e_{S_1}, \ldots, \e_{S_k})$ for the flags $\S=(\varnothing \subsetneq S_1 \subsetneq \cdots \subsetneq S_k \subsetneq E)$ 
 \end{itemize}
\end{Def}

The braid fan is the decomposition of $\N_E$ determined by the \defword{braid arrangement} in $\N_E$, which consists of the hyperplanes $t_i = t_j$ for $i, j \in E$. If $|E|=n+1$, the braid fan $\Sigma_E$ is $n$-dimensional, and has a facet $\sigma_{\S} = \sigma_{\pi} = \{\t \in \N_E \, : \, t_{s_0} \geq t_{s_1} \geq \cdots \geq t_{s_n}\}$ for each complete flag $\S=(\varnothing \subsetneq \{s_0\} \subsetneq \cdots \subsetneq \{s_0,s_1, \ldots, s_{n-1}\} \subsetneq E)$, or equivalently, each bijection $\pi: [0,n] \rightarrow E$ given by $\pi(i) = s_i$. Slightly abusing terminology, we will call $\pi$ a \emph{permutation} of $E$ and write $\pi=s_0 \ldots, s_n$.
It follows that the braid fan is complete, simplicial, and unimodular.

Figure \ref{fig:braidfan} shows the braid fan $\Sigma_E$ for $E=[0,2]=\{0,1,2\}$. It is the complete fan in $\N_E$ cut out by the braid arrangement consisting of the lines $t_0=t_1, t_1=t_2,$ and $t_2=t_0$ in $\N_E$.

\begin{figure}[H]
\centering
\scalebox{1.2}
{
\begin{tikzpicture}[root/.style={circle,draw,inner sep=1.2pt,fill=black},every node/.style={scale=0.7}]
\draw[style=thin,color=gray] (2,0) -- (-2,0);
\draw[style=thin,color=gray] (-1,-1.72) -- (1,1.72);
\draw[style=thin,color=gray] (1,-1.72) -- (-1,1.72);
\node at (0,-1.5) {$t_2>t_0>t_1$};
\node at (0,1.5) {$t_1>t_0>t_2$};
\node at (1.3,-0.7) {$t_0>t_2>t_1$};
\node at (-1.3,-0.7) {$t_2>t_1>t_0$};
\node at (1.3,0.7) {$t_0>t_1>t_2$};
\node at (-1.3,0.7) {$t_1>t_2>t_0$};
\node at (0,0) {};
\node at (.8,-1.42) [label={[label distance=5pt]300:{$\e_{02} = (1,0,1)$}}] {};
\node at (.8,1.42) [label={[label distance=5pt]60:{$\e_{01} = (1,1,0)$}}] {};
\node at (-.8,-1.42) [label={[label distance=5pt]240:{$(0,0,1)=\e_2$}}] {};
\node at (-.8,1.42) [label={[label distance=5pt]120:{$(0,1,0)=\e_1$}}] {};
\node at (1.9,0) [label=right:{$\e_0=(1,0,0)$}] {};
\node at (-1.9,0) [label=left:{ $ \,\,\, (0,1,1)=\e_{12}$}] {};
\node (1) at (3,0) {};
\node (2) at (4.5,0) {};
\end{tikzpicture}
}
\caption{The braid fan $\Sigma_{[0,2]}$.}\label{fig:braidfan}
\end{figure}
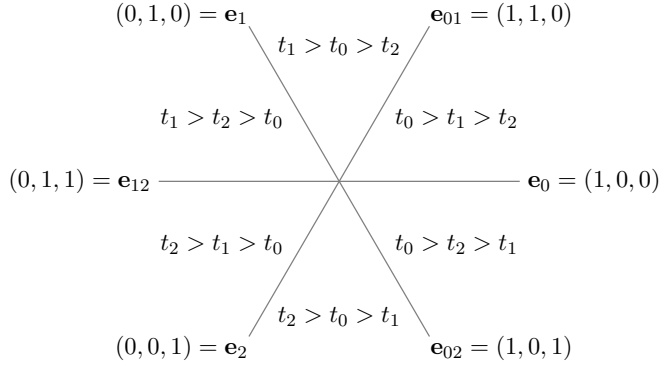

\noindent 

We will return to this picture many times in what follows; the reader may wish to keep it within reach. We will call its toric variety and Chow ring the \defword{permutahedral variety} and the \defword{permutahedral Chow ring}.

\subsection{The Chow ring as a quotient of a polynomial ring}\label{sec:A}

For the remainder of Section \ref{sec:toric}, $\Sigma$ will be a simplicial rational fan in $\N = \RR \otimes \N_{\ZZ}$. 

\medskip

\noindent \textbf{\textsf{The Chow ring.}}
The \defword{Chow ring} of $\Sigma$ is the  graded algebra 
\[
A(\Sigma) \coloneqq S(\Sigma)/(I(\Sigma)+J(\Sigma)),
\]
where
\begin{eqnarray*}
S(\Sigma) &=& \RR[x_\rho \, : \, \rho \text{ is a ray of } \Sigma]/(I(\Sigma)+J(\Sigma)), \\
I(\Sigma) &=& \langle{ x_{\rho_1} \cdots x_{\rho_k} \, : \, \rho_1, \ldots, \rho_k \text{ do not generate a cone of } \Sigma \rangle},  \\
J(\Sigma) &=& \langle{\sum_{\rho \text{ ray of } \Sigma} \ell(\e_\rho) x_\rho \, : \, \ell  \text{ is a linear function on } \N \rangle}.
\end{eqnarray*}
The ideal $I(\Sigma)$ is called the \defword{Stanley-Reisner ideal} of $\Sigma$ and $S(\Sigma)/I(\Sigma)$ is called its \defword{Stanley-Reisner ring}. In $J(\Sigma)$, it is sufficient to let $\ell$ range over a basis of the space $\N^\vee$ of linear functions on $\N$.

\begin{example} (The Chow ring $A(\Sigma_{[0,2]})$.)
Let us compute the Chow ring of the braid fan $\Sigma_E$ for $E=[0,2]$. 
We have
\begin{eqnarray*}
S(\Sigma_E) &=& \RR[x_0, x_1, x_2, x_{01}, x_{02}, x_{12}] \\
I(\Sigma_E) &=& \langle x_i x_j : i \neq j \rangle + \langle x_i x_{jk} : i,j,k \text{ distinct} \rangle + \langle x_{ij} x_{jk} : i,j,k \text{ distinct} \rangle \\
J(\Sigma_E) &=& \langle (x_0+x_{02}) - (x_1+x_{12}), (x_0+x_{01}) - (x_2+x_{12}) \rangle
\end{eqnarray*}
where we use $t_0-t_1$ and $t_0-t_2$ as a basis for $\N^\vee$ in the description of $J(\Sigma_E)$. 

We claim that $A=A(\Sigma_E)$ has degree $2$ and
\[
A^0 = \RR\{1\} \cong \RR^1, \qquad A^1 = \RR\{x_0, x_1, x_2, x_{12}\} \cong \RR^4, \qquad A^2 = \RR\{x_0x_{01}\} \cong \RR.
\] 
The description of $A^0$ is clear. 
The description of $A^1$ follows from the two linear relations in $J(\Sigma_E)$ that express $x_{01}$ and $x_{02}$ in terms of the four chosen generators.
To compute $A^2$, notice that 
\[
x_0^2 = x_0(x_2+x_{12}-x_{01}) = -x_0x_{01}, \qquad x_{01}^2 = x_{01}(x_2+x_{12}-x_{0}) = -x_0x_{01},
\]
and similarly for the squares of the other terms $x_ix_{ij}$. This implies that 
\begin{equation}\label{eq:[0,2]}
-x_0^2 = -x_1^2 = -x_2^2 = -x_{01}^2 = -x_{02}^2 = -x_{12}^2 = x_ix_{ij} \text{ for all } i \neq j.
\end{equation}
Thus $A^2$ is indeed generated by $x_0x_{01}$, and we have an isomorphism
\[
\deg: A^2 \simeq \RR,  \qquad \deg(x_ix_{ij}) = 1 \text{ for all facets $\sigma_{i \subset ij}$ of $\Sigma_{[0,2]}$}.
\]
Any monomial of degree $3$ can be reduced via \eqref{eq:[0,2]} to a square free monomial of degree $3$, which is in $I(\Sigma_E)$ and hence vanishes in $A(\Sigma_E)$.
\end{example}

\medskip

\noindent \textbf{\textsf{Computing degrees.}} When $\Sigma$ is complete, the Chow ring $A(\Sigma)$ is graded of degree $n$, and there is a canonical \defword{degree map}  $\deg: A^n(\Sigma) \simeq \RR$. If $\Sigma$ is unimodular, this map is characterized by the property that the degree of any facet monomial is 1: $\deg(x_{\sigma}) = 1$ for any facet $\sigma$, where $x_\sigma = \prod_{\rho \text{ ray}} x_\rho$. 
Any $f \in A^n(\Sigma)$ can be expressed as a linear combination of facet monomials \cite[Prop. 5.5]{AdiprasitoHuhKatz}, and this expression gives the degree of $f$.

\medskip

\noindent \textbf{\textsf{The hyperplane and reciprocal hyperplane classes.}}
We will pay special attention to
two special elements $\alpha, \beta$ in the degree one piece $A^1(\Sigma_E)$ of the permutahedral Chow ring:
\[
\alpha\coloneqq \alpha_i = \sum_{i \in S} x_S, \qquad \beta\coloneqq \beta_i = \sum_{i \notin S} x_S, \qquad \text{ for } i \in E.
\]
We invite the reader to check that these do not depend on the choice of $i \in E$.

\begin{example} (The degree of $\alpha \beta$ in $A(\Sigma_{[0,2]})$.)
For $E=[0,2]$ we have
\[
\begin{array}{rcccrcc}
\alpha 
\quad = \quad \alpha_0&=&x_0+x_{01}+x_{02} &\qquad  & \beta \quad =\quad  \beta_0&=&x_1+x_2+x_{12}  \\
\quad =\quad \alpha_1&=&x_1+x_{01}+x_{12} &\qquad&   \quad =\quad  \beta_1&=&x_0+x_2+x_{02}  \\
\quad =\quad \alpha_2&=&x_2+x_{02}+x_{12} & \qquad&  \quad =\quad  \beta_2&=&x_0+x_1+x_{01}. \\
\end{array}
\]
Let us compute the intersection degree of $\alpha$ and $\beta$. Using the relations in the Chow ring, we can write
\[
\alpha \beta = \alpha_0\beta_0 =  (x_0+x_{01}+x_{02})(x_1+x_2+x_{12}) = x_1x_{01}+x_2x_{02}.
\]
This implies that
\[
\deg (\alpha \beta) = 2.
\]
Note that a different choice of representatives, such as $\alpha_0 \beta_1$, leads to a more complicated computation. 
\end{example}

\subsection{The Chow ring in terms of piecewise polynomials}\label{sec:PP}

\noindent \textbf{\textsf{The Chow ring.}}
A \defword{piecewise polynomial} on $\Sigma$ is a continuous function on $\N$ whose restriction to each cone in $\Sigma$ agrees with a polynomial function.
Let $\PP(\Sigma)$ be the ring of piecewise polynomials on $\Sigma$, with pointwise addition and multiplication. Let $\langle \N^\vee \rangle$ be the ideal of $\PP(\Sigma)$ generated by the set $\N^\vee$ of (global) linear functions on $\N$. Thanks to work of Billera \cite{Billera},
the Chow ring of $\Sigma$ can be described as:
\[
A(\Sigma) \cong \PP(\Sigma)/\langle \N^\vee \rangle.
\]

\noindent \textbf{\textsf{The dictionary.}}
Billera \cite{Billera} constructed an isomorphism from the Stanley-Reisner ring $S(\Sigma)/I(\Sigma)$  to the algebra $\PP(\Sigma)$ of continuous piecewise polynomial functions on $\Sigma$, by identifying
  the variable $x_\rho$ with
  the piecewise linear  \defword{Courant function} on $\Sigma$ determined by the condition
\[
x_\rho(\e_{\rho'})=\begin{cases} 1, & \text{if $\rho$ is equal to $\rho'$}, \\ 0, &  \text{if $\rho$ is not equal to $\rho'$,} \end{cases} \qquad \text{ for each ray } \rho \text{ of } \Sigma.
\]
Conversely, this isomorphism identifies a piecewise linear function $\ell \in \PP(\Sigma)$ on $\Sigma$ with the linear form
\[
\ell=\sum_{\rho \text{ ray}} \ell(\e_\rho) x_\rho,
\]
and allows us to regard  the elements of $A(\Sigma)$ as equivalence classes of piecewise polynomial functions on $\Sigma$, modulo the
linear functions on $\Sigma$.

\begin{example} (The ring $A(\Sigma_{[0,2]})$)
Let us carry out this computation for the braid arrangement $\Sigma_{[0,2]}$, referring to Figure \ref{fig:braidfan}. The Courant functions representing the ray variables $x_0, x_1, x_2, x_{01}, x_{02}, x_{12}$ of the previous section are the following, where $t_{ij} \coloneqq  t_i-t_j$:

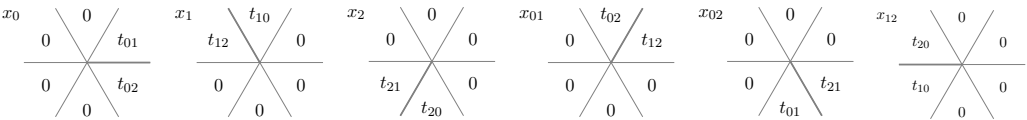
\begin{figure}[H]
\centering
\scalebox{.42}
{
\begin{tikzpicture}[root/.style={circle,draw,inner sep=1.2pt,fill=black},every node/.style={scale=1.5}]
\draw[style=thin,color=gray] (2,0) -- (-2,0);
\draw[style=ultra thick,color=gray] (2,0) -- (0,0);
\draw[style=thin,color=gray] (-1,-1.72) -- (1,1.72);
\draw[style=thin,color=gray] (1,-1.72) -- (-1,1.72);
\node at (0,-1.5) {$0$};
\node at (0,1.5) {$0$};
\node at (1.3,-0.7) {$t_{02}$};
\node at (-1.3,-0.7) {$0$};
\node at (1.3,0.7) {$t_{01}$};
\node at (-1.3,0.7) {$0$};
\node at (-1.7,1.5) [label=left:{{$x_0$}}] {};
\end{tikzpicture}
}
\scalebox{.42}
{
\begin{tikzpicture}[root/.style={circle,draw,inner sep=1.2pt,fill=black},every node/.style={scale=1.5}]
\draw[style=thin,color=gray] (2,0) -- (-2,0);
\draw[style=ultra thick,color=gray] (-1,1.72) -- (0,0);
\draw[style=thin,color=gray] (-1,-1.72) -- (1,1.72);
\draw[style=thin,color=gray] (1,-1.72) -- (-1,1.72);
\node at (0,-1.5) {$0$};
\node at (0,1.5) {$t_{10}$};
\node at (1.3,-0.7) {$0$};
\node at (-1.3,-0.7) {$0$};
\node at (1.3,0.7) {$0$};
\node at (-1.3,0.7) {$t_{12}$};
\node at (-1.7,1.5)[label=left:{{$x_1 $}}] {};
\end{tikzpicture}
}
\scalebox{.42}
{
\begin{tikzpicture}[root/.style={circle,draw,inner sep=1.2pt,fill=black},every node/.style={scale=1.5}]
\draw[style=thin,color=gray] (2,0) -- (-2,0);
\draw[style=ultra thick,color=gray] (-1,-1.72) -- (0,0);
\draw[style=thin,color=gray] (-1,-1.72) -- (1,1.72);
\draw[style=thin,color=gray] (1,-1.72) -- (-1,1.72);
\node at (0,-1.5) {$t_{20}$};
\node at (0,1.5) {$0$};
\node at (1.3,-0.7) {$0$};
\node at (-1.3,-0.7) {$t_{21}$};
\node at (1.3,0.7) {$0$};
\node at (-1.3,0.7) {$0$};
\node at (-1.7,1.5) [label=left:{{$x_2$}}] {};
\end{tikzpicture}
}
\scalebox{.42}
{
\begin{tikzpicture}[root/.style={circle,draw,inner sep=1.2pt,fill=black},every node/.style={scale=1.5}]
\draw[style=thin,color=gray] (2,0) -- (-2,0);
\draw[style=ultra thick,color=gray] (1,1.72) -- (0,0);
\draw[style=thin,color=gray] (-1,-1.72) -- (1,1.72);
\draw[style=thin,color=gray] (1,-1.72) -- (-1,1.72);
\node at (0,-1.5) {$0$};
\node at (0,1.5) {$t_{02}$};
\node at (1.3,-0.7) {$0$};
\node at (-1.3,-0.7) {$0$};
\node at (1.3,0.7) {$t_{12}$};
\node at (-1.3,0.7) {$0$};
\node at (-1.7,1.5) [label=left:{{$x_{01} $}}] {};
\end{tikzpicture}
}
\scalebox{.42}
{
\begin{tikzpicture}[root/.style={circle,draw,inner sep=1.2pt,fill=black},every node/.style={scale=1.5}]
\draw[style=thin,color=gray] (2,0) -- (-2,0);
\draw[style=ultra thick,color=gray] (1,-1.72) -- (0,0);
\draw[style=thin,color=gray] (-1,-1.72) -- (1,1.72);
\draw[style=thin,color=gray] (1,-1.72) -- (-1,1.72);
\node at (0,-1.5) {$t_{01}$};
\node at (0,1.5) {$0$};
\node at (1.3,-0.7) {$t_{21}$};
\node at (-1.3,-0.7) {$0$};
\node at (1.3,0.7) {$0$};
\node at (-1.3,0.7) {$0$};
\node at (-1.7,1.5) [label=left:{{$x_{02} $}}] {};
\end{tikzpicture}
}
\scalebox{.42}
{
\begin{tikzpicture}[root/.style={circle,draw,inner sep=1.2pt,fill=black},every node/.style={scale=1.25}]
\draw[style=thin,color=gray] (2,0) -- (-2,0);
\draw[style=ultra thick,color=gray] (0,0) -- (-2,0);
\draw[style=thin,color=gray] (-1,-1.72) -- (1,1.72);
\draw[style=thin,color=gray] (1,-1.72) -- (-1,1.72);
\node at (0,-1.5) {$0$};
\node at (0,1.5) {$0$};
\node at (1.3,-0.7) {$0$};
\node at (-1.3,-0.7) {$t_{10}$};
\node at (1.3,0.7) {$0$};
\node at (-1.3,0.7) {$t_{20}$};
\node at (-1.7,1.5) [label=left:{{$x_{12} $}}] {};
\end{tikzpicture}
}
\caption{The Courant functions $x_0, x_1, x_2, x_{01}, x_{02}, x_{12}$ on $\Sigma_{[0,2]}$. Each function $x_S$ equals $1$ on the marked primitive ray $\e_S$ and $0$ on the others.}\label{FigureBipermutohedralFan}
\end{figure}

As we saw in the previous section, $A^0$ is generated by the constant function $1$, $A^1$ is generated by $x_0, x_1, x_2, x_{01}$, and $A^2$ is generated by
\begin{figure}[H]
\centering
\scalebox{.42}
{
\begin{tikzpicture}[root/.style={circle,draw,inner sep=1.2pt,fill=black},every node/.style={scale=1.5}]
\draw[style=thin,color=gray] (2,0) -- (-2,0);
\draw[style=thin,color=gray] (-1,-1.72) -- (1,1.72);
\draw[style=thin,color=gray] (1,-1.72) -- (-1,1.72);
\draw[style=ultra thick,color=gray] (0,0) -- (2,0);
\draw[style=ultra thick,color=gray] (0,0) -- (1,1.72);
\node at (-1.9,0) [label=left:{$x_0x_{01} =$}] {};
\node at (0,-1.5) {$0$};
\node at (0,1.5) {$0$};
\node at (1.3,-0.7) {$0$};
\node at (-1.3,-0.7) {$0$};
\node at (1.3,0.7) {$t_{01}t_{12}$};
\node at (-1.3,0.7) {$0$};
\node at (.8,1.42) [label={[label distance=5pt]60:{$\e_{01}$}}] {};
\node at (1.9,0) [label=right:{$\e_0$.}] {};
\end{tikzpicture}
}
\end{figure}

\noindent
This expression for the generator $x_0x_{01}$ is supported on the chamber $\cone\{\e_0, \e_{01}\} = \{\t \in \N_E \, : \, t_0>t_1>t_2\}$. Its unique non-zero polynomial $t_{01}t_{12}$ is the product of the linear forms $t_0-t_1$ and $t_1-t_2$ defining the inequalities of the chamber. 

It is instructive to double check that two adjacent chambers (and hence any two chambers) give the same generator of $A^2$. The neighbor chamber  
$\cone\{\e_0, \e_{02}\} = \{\t \in \N_E \, : \, t_0>t_2>t_1\}$
separated by the wall $t_{12}=0$, gives generator $x_0x_{02}$. Their difference is
\begin{figure}[H]
\centering
\scalebox{.42}
{
\begin{tikzpicture}[root/.style={circle,draw,inner sep=1.2pt,fill=black},every node/.style={scale=1.5}]
\draw[style=thin,color=gray] (2,0) -- (-2,0);
\draw[style=thin,color=gray] (-1,-1.72) -- (1,1.72);
\draw[style=thin,color=gray] (1,-1.72) -- (-1,1.72);
\draw[style=ultra thick,color=gray] (0,0) -- (2,0);
\node at (-1.9,0) [label=left:{$x_0x_{01} - x_0x_{02} =$}] {};
\node at (0,-1.5) {$0$};
\node at (0,1.5) {$0$};
\node at (1.3,-0.7) {$-t_{02}t_{21}$};
\node at (-1.3,-0.7) {$0$};
\node at (1.3,0.7) {$t_{01}t_{12}$};
\node at (-1.3,0.7) {$0$};
\end{tikzpicture}
}
\scalebox{.42}
{
\begin{tikzpicture}[root/.style={circle,draw,inner sep=1.2pt,fill=black},every node/.style={scale=1.5}]
\draw[style=thin,color=gray] (2,0) -- (-2,0);
\draw[style=thin,color=gray] (-1,-1.72) -- (1,1.72);
\draw[style=thin,color=gray] (1,-1.72) -- (-1,1.72);
\draw[style=ultra thick,color=gray] (0,0) -- (2,0);
\node at (-1.9,0) [label=left:{= $t_{12} \, \, \cdot$ }] {};
\node at (0,-1.5) {$0$};
\node at (0,1.5) {$0$};
\node at (1.3,-0.7) {$t_{02}$};
\node at (-1.3,-0.7) {$0$};
\node at (1.3,0.7) {$t_{01}$};
\node at (-1.3,0.7) {$0$};
\node at (1.9,0) [label=right:{ = 0}] {};
\end{tikzpicture}
}
\end{figure}

\noindent
since it is the product of the linear function $t_{12}$ of the wall separating them and a piecewise polynomial function: we have $t_{01}=0$ on $\e_{01}$, $t_{01}=t_{02}$ on $\e_0$, and $t_{02}=0$ on $\e_{02}$.
This generalizes to any two neighbor chambers in any braid fan, and further, in any simplicial rational fan.
\end{example}

\noindent \textbf{\textsf{Computing degrees.}}
There is a very elegant way to compute the degree of an element $f \in A^n(\Sigma)$  given by a piecewise polynomial $f=(f_\sigma \, : \, \sigma \in \Sigma(n)$). 
To describe it, we first associate a rational function to each facet $\sigma$ of $\Sigma$.
If $\sigma$ is simplicial and unimodular, it is generated by $n$ inequalities $f_1(x) \geq 0, \ldots, f_n(x) \geq 0$, where $\{f_1, \ldots, f_n\}$ is the basis dual to the rays generating $\sigma$. This determines a rational function $e_\sigma \coloneqq 1/(f_1 \cdots f_n)$ in $\Sym^\pm(N^\vee)$. 
In general, we can triangulate $\sigma$ into simplicial unimodular cones $\sigma_1, \ldots, \sigma_n$ and define $e_\sigma \coloneqq  e_{\sigma_1}+\cdots+e_{\sigma_n}$, which turns out to be independent of the triangulation \cite{BrionEquivariant}. We then have
\[ 
\deg(f) = \sum_{\sigma \in \Sigma(n)} e_\sigma f_\sigma.
\]
It is pleasant and not a priori obvious that this is always a constant, after significant cancellation. It is not so difficult to prove it, though, by verifying that the above formula gives $\deg(x_\sigma)=1$ for every facet monomial and $0$ for every other square-free monomial.

\medskip

\noindent \textbf{\textsf{The hyperplane and reciprocal hyperplane classes.}}
The elements $\alpha$ and $\beta$ of the permutahedral Chow ring can be described by the following piecewise linear functions, for any $i \in E$:
\[
\alpha = \alpha_i = \max(t_i - t_j \, : \, j \in E), \qquad \beta = \beta_i = \max(t_j - t_i \, : \, j \in E). 
\]
For any $i \neq i'$ the function $\alpha_i-\alpha_{i'} = t_i-t_{i'}$ is linear, and hence in $ \N^\vee$, so $\alpha$ is well-defined.\footnote{It is tempting but incorrect to think that $t_i$ is linear so we can write $\alpha = max(- t_j \, : \, j \in E)$: in fact 
$t_i$ is not even a well defined function on the ambient space $\N = \RR^{[0,2]}/\RR\e_{[0,2]}$.}
To verify the formula for $\alpha_i$, notice that the value of $\max(t_i-t_j \, : \, j \in E)$ on $\e_S$ is $1$ if $i \in S$ and $0$ if $i \notin S$.
A similar argument works for $\beta$.

\begin{example} (The degree of $\alpha \beta$ in $A(\Sigma_{[0,2])}$.)
The special element $\alpha \in A(\Sigma_E)$ is given by the expressions 
$\alpha_0=x_0+x_{01}+x_{02}$,  
$\alpha_1=x_1+x_{01}+x_{12}$,  and
$\alpha_2=x_2+x_{02}+x_{12}$, which give: 
\begin{figure}[H]
\centering
\scalebox{.42}
{
\begin{tikzpicture}[root/.style={circle,draw,inner sep=1.2pt,fill=black},every node/.style={scale=1.5}]
\draw[style=thin,color=gray] (2,0) -- (-2,0);
\draw[style=thin,color=gray] (-1,-1.72) -- (1,1.72);
\draw[style=thin,color=gray] (1,-1.72) -- (-1,1.72);
\node at (-1.9,0) [label=left:{$\alpha = $}] {};
\node at (0,-1.5) {$t_{01}$};
\node at (0,1.5) {$t_{02}$};
\node at (1.3,-0.7) {$t_{01}$};
\node at (-1.3,-0.7) {$0$};
\node at (1.3,0.7) {$t_{02}$};
\node at (-1.3,0.7) {$0$};
\end{tikzpicture}
}
\scalebox{.42}
{
\begin{tikzpicture}[root/.style={circle,draw,inner sep=1.2pt,fill=black},every node/.style={scale=1.5}]
\draw[style=thin,color=gray] (2,0) -- (-2,0);
\draw[style=thin,color=gray] (-1,-1.72) -- (1,1.72);
\draw[style=thin,color=gray] (1,-1.72) -- (-1,1.72);
\node at (-1.9,0) [label=left:{ $=$}] {};
\node at (0,-1.5) {$0$};
\node at (0,1.5) {$t_{12}$};
\node at (1.3,-0.7) {$0$};
\node at (-1.3,-0.7) {$t_{10}$};
\node at (1.3,0.7) {$t_{12}$};
\node at (-1.3,0.7) {$t_{10}$};
\end{tikzpicture}
}
\scalebox{.42}
{
\begin{tikzpicture}[root/.style={circle,draw,inner sep=1.2pt,fill=black},every node/.style={scale=1.5}]
\draw[style=thin,color=gray] (2,0) -- (-2,0);
\draw[style=thin,color=gray] (-1,-1.72) -- (1,1.72);
\draw[style=thin,color=gray] (1,-1.72) -- (-1,1.72);
\node at (-1.9,0) [label=left:{$=$}] {};
\node at (0,-1.5) {$t_{21}$};
\node at (0,1.5) {$0$};
\node at (1.3,-0.7) {$t_{21}$};
\node at (-1.3,-0.7) {$t_{20}$};
\node at (1.3,0.7) {$0$};
\node at (-1.3,0.7) {$t_{20}$};
\node at (1.9,0) [label=right:{ $.$}] {};
\end{tikzpicture}
}
\end{figure}

\noindent These look different, but they are equal modulo global linear functions on $N$: the first two differ by $t_{01}$ and the latter two differ by $t_{12}$. Similarly, there are three natural piecewise linear representatives for $\beta$, namely $\beta_0, \beta_1, \beta_2$.
Let's compute the degree of $\alpha \beta$ in two ways, referring to Figure \ref{fig:braidfan} again. Since

\begin{figure}[H]
\centering
\scalebox{.42}
{
\begin{tikzpicture}[root/.style={circle,draw,inner sep=1.2pt,fill=black},every node/.style={scale=1.5}]
\draw[style=thin,color=gray] (2,0) -- (-2,0);
\draw[style=thin,color=gray] (-1,-1.72) -- (1,1.72);
\draw[style=thin,color=gray] (1,-1.72) -- (-1,1.72);
\node at (-1.9,0) [label=left:{$\alpha_0 \beta_0 = $}] {};
\node at (0,-1.5) {$t_{01}$};
\node at (0,1.5) {$t_{02}$};
\node at (1.3,-0.7) {$t_{01}$};
\node at (-1.3,-0.7) {$0$};
\node at (1.3,0.7) {$t_{02}$};
\node at (-1.3,0.7) {$0$};
\end{tikzpicture}
}
\scalebox{.42}
{
\begin{tikzpicture}[root/.style={circle,draw,inner sep=1.2pt,fill=black},every node/.style={scale=1.5}]
\draw[style=thin,color=gray] (2,0) -- (-2,0);
\draw[style=thin,color=gray] (-1,-1.72) -- (1,1.72);
\draw[style=thin,color=gray] (1,-1.72) -- (-1,1.72);
\node at (0,-1.5) {$t_{20}$};
\node at (0,1.5) {$t_{10}$};
\node at (1.3,-0.7) {$0$};
\node at (-1.3,-0.7) {$t_{20}$};
\node at (1.3,0.7) {$0$};
\node at (-1.3,0.7) {$t_{10}$};
\end{tikzpicture}
}
\scalebox{.42}
{
\begin{tikzpicture}[root/.style={circle,draw,inner sep=1.2pt,fill=black},every node/.style={scale=1.5}]
\draw[style=thin,color=gray] (2,0) -- (-2,0);
\draw[style=thin,color=gray] (-1,-1.72) -- (1,1.72);
\draw[style=thin,color=gray] (1,-1.72) -- (-1,1.72);
\node at (-1.9,0) [label=left:{$=$}] {};
\node at (0,-1.5) {$t_{01}t_{20}$};
\node at (0,1.5) {$t_{02}t_{10}$};
\node at (1.3,-0.7) {$0$};
\node at (-1.3,-0.7) {$0$};
\node at (1.3,0.7) {$0$};
\node at (-1.3,0.7) {$0$};
\node at (1.9,0) [label=right:{, \qquad \qquad }] {};
\end{tikzpicture}
}
\scalebox{.42}
{
\begin{tikzpicture}[root/.style={circle,draw,inner sep=1.2pt,fill=black},every node/.style={scale=1.5}]
\draw[style=thin,color=gray] (2,0) -- (-2,0);
\draw[style=thin,color=gray] (-1,-1.72) -- (1,1.72);
\draw[style=thin,color=gray] (1,-1.72) -- (-1,1.72);
\node at (-1.9,0) [label=left:{$\alpha_1 \beta_0 = $}] {};
\node at (0,-1.5) {$0$};
\node at (0,1.5) {$t_{10}t_{12}$};
\node at (1.3,-0.7) {$0$};
\node at (-1.3,-0.7) {$t_{10}t_{20}$};
\node at (1.3,0.7) {$0$};
\node at (-1.3,0.7) {$t_{10}^2$};
\node at (1.9,0) [label=right:{ , }] {};
\end{tikzpicture}
}
\end{figure}

\noindent we have that $\deg(\alpha\beta)$ equals
\[
 \frac{t_{02}t_{10}}{t_{02}t_{10}} + 
 \frac{t_{01}t_{20}}{t_{01}t_{20}} = 2
 \text{ and } \qquad
 \frac{t_{10}t_{12}}{t_{02}t_{10}} +
  \frac{t_{10}^2}{t_{12}t_{20}} + 
  \frac{t_{10}t_{20}}{t_{21}t_{10}} = 2,
\]
where the first computation is immediate and the second involves a fun cancellation. 
\end{example}

\subsection{The Chow ring in terms of Minkowski weights}\label{sec:MW}



\noindent \textbf{\textsf{The Chow ring.}}
A \defword{$k$-dimensional Minkowski weight} on $\Sigma$ is a real-valued function
$\omega$ on the set $\Sigma(k)$ of $k$-dimensional cones that satisfies the \defword{balancing condition}: For every $(k-1)$-dimensional cone $\tau$ in $\Sigma$,
\[
\text{$\sum_{\tau \subset \sigma} \omega(\sigma) \mathbf{e}_{\sigma/\tau}=0$ in the quotient space $\N/ \span(\tau)$,}
\]
where $\mathbf{e}_{\sigma/\tau}$ is the primitive generator of the ray $( \sigma+ \span(\tau) ) / \span(\tau)$.
We say that $w$ is \defword{positive} if $w(\sigma)$ is positive for every $\sigma$ in $\Sigma(k)$.
We write $\MW_k(\Sigma)$ for the space of $k$-dimensional Minkowski weights on $\Sigma$, and set 
$
\MW(\Sigma)=\bigoplus_{k \ge 0} \MW_k(\Sigma).
$

The product in $\MW(\Sigma)$ is given by the following \defword{fan displacement rule}. If $X_1$ and $X_2$ are Minkowski weights of codimension $k$ and $\ell$ on $\Sigma$, then their product is defined to be the \defword{stable intersection}
\[
X_1 \cdot X_2 \coloneqq \lim_{\epsilon \rightarrow 0} X_1 \cdot (X_2 + \epsilon \v)
\]
for any vector $\v \in \N$ such that $X_1$ and $X_2 + \epsilon \v$ intersect transversally for sufficiently small $\epsilon > 0$. The facets of $X_1 \cdot X_2$ are the $(k+\ell)$--codimensional intersections of a facet of $X_1$ and a facet of $X_2$. The weight of a facet $\tau$ of $X_1 \cdot X_2$ is
\[
w(\tau) = \sum_{\sigma_1, \sigma_2} w(\sigma_1) w(\sigma_2) [\ZZ^n : L_{\ZZ}(\sigma_1) + L_{\ZZ}(\sigma_2)],
\]
summing over the facets $\sigma_1$ and $\sigma_2$ of $X_1$ and $X_2$ respectively such that $\tau = \sigma_1 \cap \sigma_2$ and $\sigma_1 \cap (\sigma_2 + \epsilon \v) \neq 0$ for small $\epsilon >0$. 
It is non-trivial that the construction above is independent of the choice of a (generic) vector $\v$, and that it is also a Minkowski weight, that is, it satisfies the balancing condition \cite{FultonSturmfels, JensenYu}.

When $\Sigma$ is complete, 
Fulton and Sturmfels \cite{FultonSturmfels} proved that
\[
A(\Sigma) \cong \MW(\Sigma)
\]
so understanding the Chow ring of $\Sigma$ is equivalent to understanding  Minkowski weights on $\Sigma$ and their stable intersections.

%

\medskip

\noindent \textbf{\textsf{The dictionary.}}
For $\Sigma$ complete, 
Katz and Payne \cite{KatzPayne} described the canonical\footnote{This is canonical in the sense that $\PP(\Sigma) \cong A_T(X_\Sigma)$ and $\MW(\Sigma) \cong A(X_\Sigma)$ are isomorphic to the equivariant and the ordinary Chow cohomology rings of the toric variety $X_\Sigma$, respectively, and there is a canonical map $ A_T(X_\Sigma) \rightarrow A(X_\Sigma)$.} map from $\PP(\Sigma)$ to $\MW(\Sigma)$ that descends to an isomorphism $A^k(\Sigma) \cong \MW_{n-k}(\Sigma)$. 
We focus on a different description for a special case: when $f \in \PP^1(\Sigma)$ is a piecewise linear function that is convex, that is, $f((\x+\y)/2) \leq (f(\x)+f(\y))/2$ for all $\x,\y \in \N$. In this case, $f$ can be written as a \defword{tropical polynomial}; that is, the maximum of a finite number of linear functions:
\[
f(\x) = \max\{\v_1(\x), \ldots, \v_m(\x)\} \qquad \text{ for }  \v_1, \ldots, \v_m \in \N^\vee.
\]
The \defword{corner locus}, where this function is not linear, is the \defword{tropical hypersurface}:
\[
\trop f = \{\x \in \N \, : \, \max_{1 \leq i \leq m} \{\v_i(\x)\} \text{ is achieved at least twice} \}.
\]
This is the $(n-1)$-skeleton of the normal fan of the \defword{Newton polytope} $\Newt(f) = \conv(\v_1, \ldots, \v_m)$. It turns into a balanced fan with a natural choice of weights: for each facet $F$ of $\trop f$ the weight $w(F) = \ell(F^\vee)$ equals the lattice length of the corresponding edge of $\Newt(f)$. This balanced fan is the Minkowski weight in $\MW_{n-1}(\Sigma)$ corresponding to $f$. 
For details, see \cite{MaclaganSturmfels, McMullen, MikhalkinRau}.

\begin{example} (The Chow ring $A(\Sigma_{[0,2])})$.)
Let $\Sigma=\Sigma[0,2]$. For $k=0$ the balancing condition is vacuous and a Minkowski weight is a choice of a weight on the origin. For $k=1$, we need to put a weight on each of the six rays so that the weighted sum of the rays is $0$. The four choices of weight below generate all others. For $k=2$ we need weights on each maximal cone of $\Sigma$. Each ray $\tau$  is in two cones $\sigma_1$ and $\sigma_2$ which satisfy $\e_{\sigma_1/\tau} = -\e_{\sigma_2/\tau}$, so the balancing condition says that $w(\sigma_1)=w(\sigma_2)$, and hence all weights are equal. Thus $\MW(\Sigma)$ is spanned by the following Minkowski weights:

\begin{figure}[H]
\centering
\scalebox{.4}
{
\begin{tikzpicture}[root/.style={circle,draw,inner sep=1.2pt,fill=black},every node/.style={scale=1.5}]
\node[circle,draw,inner sep=1.2pt,fill=black] at (0,0)[label=above left:{$1$}] {};
\node at (0,-1.5) {$$};
\node at (0,1.5) {$$};
\node at (1.3,-0.7) {$$};
\node at (-1.3,-0.7) {$$};
\node at (1.3,0.7) {$$};
\node at (-1.3,0.7) {$$};
\end{tikzpicture}
}
\qquad \qquad
\scalebox{.4}
{
\begin{tikzpicture}[root/.style={circle,draw,inner sep=1.2pt,fill=black},every node/.style={scale=1.5}]
\draw[style=ultra thick,color=gray] (-1,1.72) -- (1,-1.72);
\node[circle,draw,inner sep=1.2pt,fill=black] at (0,0) {};
\node at (0,-1.5) {$1$};
\node at (0,1.5) {$$};
\node at (1.3,-0.7) {$$};
\node at (-1.3,-0.7) {$$};
\node at (1.3,0.7) {$$};
\node at (-1.3,0.7) {$1$};
\end{tikzpicture}
}
\scalebox{.4}
{
\begin{tikzpicture}[root/.style={circle,draw,inner sep=1.2pt,fill=black},every node/.style={scale=1.5}]
\draw[style=ultra thick,color=gray] (-1,-1.72) -- (1,1.72);
\node[circle,draw,inner sep=1.2pt,fill=black] at (0,0) {};
\node at (0,-1.5) {$1$};
\node at (0,1.5) {$$};
\node at (1.3,-0.7) {$$};
\node at (-1.3,-0.7) {$$};
\node at (1.3,0.7) {$1$};
\node at (-1.3,0.7) {$$};
\end{tikzpicture}
}
\scalebox{.4}
{
\begin{tikzpicture}[root/.style={circle,draw,inner sep=1.2pt,fill=black},every node/.style={scale=1.5}]
\draw[style=ultra thick,color=gray] (2,0) -- (-2,0);
\node[circle,draw,inner sep=1.2pt,fill=black] at (0,0) {};
\node at (0,-1.5) {};
\node at (0,1.5) {};
\node at (1.3,-0.7) {$1$};
\node at (-1.3,-0.7) {$1$};
\node at (1.3,0.7) {};
\node at (-1.3,0.7) {};
\end{tikzpicture}
}
\scalebox{.4}
{
\begin{tikzpicture}[root/.style={circle,draw,inner sep=1.2pt,fill=black},every node/.style={scale=1.5}]
\draw[style=ultra thick,color=gray] (2,0) -- (0,0);
\draw[style=ultra thick,color=gray] (-1,1.72) -- (0,0);
\draw[style=ultra thick,color=gray] (-1,-1.72) -- (0,0);
\node[circle,draw,inner sep=1.2pt,fill=black] at (0,0) {};
\node at (0,-1.5) {};
\node at (0,1.5) {$1$};
\node at (1.3,-0.7) {$1$};
\node at (-1.3,-0.7) {};
\node at (1.3,0.7) {};
\node at (-1.3,-0.7) {$1$};
\end{tikzpicture}
}
\qquad \qquad
\scalebox{.4}
{
\begin{tikzpicture}[root/.style={circle,draw,inner sep=1.2pt,fill=black},every node/.style={scale=1.25}]
\draw[style=ultra thick,color=gray] (2,0) -- (-2,0);
\draw[style=ultra thick,color=gray] (0,0) -- (-2,0);
\draw[style=ultra thick,color=gray] (-1,-1.72) -- (1,1.72);
\draw[style=ultra thick,color=gray] (1,-1.72) -- (-1,1.72);
\node[circle,draw,inner sep=1.2pt,fill=black] at (0,0) {};
\node at (0,-1.5) {$1$};
\node at (0,1.5) {$1$};
\node at (1.3,-0.7) {$1$};
\node at (-1.3,-0.7) {$1$};
\node at (1.3,0.7) {$1$};
\node at (-1.3,0.7) {$1$};
\end{tikzpicture}
}
\end{figure}
\end{example}

\medskip

\noindent \textbf{\textsf{Computing degrees.}}
One can use the \defword{fan displacement rule} to compute the degree of a product: $X_1 \cdot X_2 = \lim_{\epsilon \rightarrow 0} X_1 \cdot (X_2 + \epsilon \v)$, where $\v \in \N$ is any vector such that $X_1$ and $X_2 + \epsilon \v$ intersect transversally for sufficiently small $\epsilon > 0$. This requires one to understand how these fans intersect by solving systems of linear equations and inequalities. Sometimes a clever  choice of $\v$ -- for example one whose coordinates increase very quickly -- can simplify the computations.

\medskip

\noindent \textbf{\textsf{The hyperplane and reciprocal hyperplane classes.}}
In $\Sigma_E = \Sigma_{[0,n]}$, the Minkowski weights of $\alpha = \alpha_i = \max(t_i - t_j \, : \, j \in E)$ and $\beta = \beta_i = \max(t_j - t_i \, : \, j \in E)$ are the $(n-1)$-skeleta of the normal fans of $\Newt(\alpha_i) = \e_i-\Delta_E$ and $\Newt(\beta_i) = \Delta_E-\e_i$ where $\Delta_E = \conv(\e_0, \ldots, \e_n)$ is the standard simplex. 
Explicitly, the facets of $\alpha$ and $\beta$ are:
\begin{eqnarray*}
\alpha: &&
\{\sigma_{i_0 \subset i_0i_1 \subset \cdots \subset i_0i_1\ldots i_{n-2}} \, : \, i_0i_1 \ldots i_{n-1}i_n \text{ permutation of } E\} \\
 \beta: && 
\{\sigma_{i_0i_1 \subset \cdots \subset i_0i_1\ldots i_{n-1} \subset i_0i_1\ldots i_{n-1}} \, : \, i_0 i_1 \ldots i_{n-1}i_n \text{ permutation of } E\}
\end{eqnarray*}
with unit weights on all facets. The supports of these fans are
\begin{eqnarray*}
|\alpha| &=& \{\t \in \N_E \, : \, \min_{i \in E} t_i \text{ is achieved at least twice}\} \\
|\beta| &=& \{\t \in \N_E \, : \, \max_{i \in E} t_i \text{ is achieved at least twice}\}.
\end{eqnarray*}
Notice that $\min_{i \in E} t_i$ is not a well defined function on $\N_E = \RR^E / \RR\e_E$, but whether  or not this minimum is achieved at least twice \textbf{is} well defined; similarly for $\beta$.

\begin{example} (The degree of $\alpha \beta$ in $A(\Sigma_{[0,2])}$.)
In $A(\Sigma_{[0,2]})$, we can draw the Minkowski weights of $\alpha$ and $\beta$ using the description obtained above. 
Alternatively, we can look at the expressions for $\alpha$ and $\beta$ as piecewise linear functions in Section \ref{sec:PP} and draw their corner loci, where the functions are not locally linear.

\begin{figure}[H]
\centering
\scalebox{.4}
{
\begin{tikzpicture}[root/.style={circle,draw,inner sep=1.2pt,fill=black},every node/.style={scale=1.5}]
\node at (-1.9,0) [label=left:{$\alpha=$}] {};
\draw[style=ultra thick,color=gray] (2,0) -- (0,0);
\draw[style=ultra thick,color=gray] (-1,1.72) -- (0,0);
\draw[style=ultra thick,color=gray] (-1,-1.72) -- (0,0);
\end{tikzpicture}
}
\quad
\scalebox{.4}
{
\begin{tikzpicture}[root/.style={circle,draw,inner sep=1.2pt,fill=black},every node/.style={scale=1.5}]
\node at (-1.9,0) [label=left:{$\beta=$}] {};
\draw[style=ultra thick,color=gray] (-2,0) -- (0,0);
\draw[style=ultra thick,color=gray] (1,-1.72) -- (0,0);
\draw[style=ultra thick,color=gray] (1,1.72) -- (0,0);
\end{tikzpicture}
}
\end{figure}

\noindent
Using these Minkowski weights, we compute the degree of $\alpha \beta$ in $A(\Sigma_{[0,2])}$ in two ways:

\begin{figure}[H]
\centering
\scalebox{.4}
{
\begin{tikzpicture}[root/.style={circle,draw,inner sep=1.2pt,fill=black},every node/.style={scale=1.5}]
\node at (-1.9,0) [label=left:{$\alpha \cap (\beta + \epsilon_1v_1)=$}] {};
\draw[style=ultra thick,color=gray] (2,0) -- (0,0);
\draw[style=ultra thick,color=gray] (-1,1.72) -- (0,0);
\draw[style=ultra thick,color=gray] (-1,-1.72) -- (0,0);
\draw[style=ultra thick,color=gray] (-2,0.6) -- (0,0.6);
\draw[style=ultra thick,color=gray] (1,-1.12) -- (0,0.6);
\draw[style=ultra thick,color=gray] (1,2.32) -- (0,0.6);
\node[circle,draw,inner sep=1.2pt,fill=black] at (-0.35,0.6) {};
\node[circle,draw,inner sep=1.2pt,fill=black] at (0.35,0) {};
\end{tikzpicture}
}
\quad
\scalebox{.4}
{
\begin{tikzpicture}[root/.style={circle,draw,inner sep=1.2pt,fill=black},every node/.style={scale=1.5}]
\node at (-1.9,0) [label=left:{$\alpha \cap (\beta + \epsilon_2v_2)=$}] {};
\draw[style=ultra thick,color=gray] (2,0) -- (0,0);
\draw[style=ultra thick,color=gray] (-1,1.72) -- (0,0);
\draw[style=ultra thick,color=gray] (-1,-1.72) -- (0,0);
\draw[style=ultra thick,color=gray] (-2.5,0) -- (-0.5,0);
\draw[style=ultra thick,color=gray] (0.5,-1.72) -- (-0.5,0);
\draw[style=ultra thick,color=gray] (0.5,1.72) -- (-0.5,0);
\node[circle,draw,inner sep=1.2pt,fill=black] at (-0.25,-0.45) {};
\node[circle,draw,inner sep=1.2pt,fill=black] at (-0.25,0.45) {};
\end{tikzpicture}
}
\end{figure}
\noindent
In each case the index of intersection is $1$, so $\deg(\alpha \beta) = 2$.

\end{example}

\subsection{The Chow ring in terms of tropical intersection}

\noindent \textbf{\textsf{The Chow ring.} }
When $\Sigma$ is complete, 
there is an alternative description of the product in the Chow ring that combines piecewise polynomials and Minkowski weights \cite{AllermannRau, KatzTropInt, Mikhalkin, Rau}.
Let $w \in \MW_k(\Sigma)$ be a Minkowski weight  on $\Sigma$, and $f \in  A^1(\Sigma)$ be a piecewise linear function (modulo global linear functions) on $\Sigma$, regarded as a codimension 1 Minkowski weight. The
Minkowski weight $f \cdot w \in \MW_{k-1}(\Sigma)$ is given by
\begin{equation} \label{eq:cap}
f \cdot w \,\, (\tau) \coloneqq  
\sum_{\stackrel{\sigma \in \Sigma^{(k)}}{\sigma \supset \tau}}  f(w(\sigma) \e_{\sigma/\tau})  - 
f \left(\sum_{\stackrel{\sigma \in \Sigma^{(k)}}{\sigma \supset \tau}}  w(\sigma) \e_{\sigma/\tau}\right)
\end{equation}
for each $(k-1)$-cone $\tau$ of $\Sigma$. 
In tropical geometry, the Minkowski weight  $f \cdot w$ is known as the \defword{divisor} $\div_w(f)$. Intuitively, it measures the non-linearity of $f$ on $w$. In particular, if $w$ is linear on $f$ locally around $\tau$, then the divisor equals $0$ at $\tau$.

\medskip

\noindent \textbf{\textsf{The dictionary.}} 
Underlying this description is an isomorphism $\MW(\Sigma) \simeq \Hom(A(\Sigma),\RR)$ given by the maps
\begin{eqnarray*}
\MW_k(\Sigma) &\simeq&  \Hom(A^k(\Sigma),\RR) \\
w  & \longmapsto & (x_\sigma \mapsto w(\sigma)/\mult(\sigma) \text{ for each $k$-face } \sigma)
\end{eqnarray*}
for any simplicial $\Sigma$ \cite{AdiprasitoHuhKatz, ArdilaDenhamHuh1}.
This isomorphism gives $\MW(\Sigma)$ the structure of a graded $A(\Sigma)$-module. 
\footnote{The map $\cdot: A(\Sigma) \times \MW(\Sigma) \rightarrow \MW(\Sigma)$ is sometimes called the \defword{cap product} $\cap$.} 
When $\Sigma$ is complete, 
we can compute a product $w_1w_2$ by regarding $w_1 \in \MW_{n-n_1}(\Sigma)$ as the image of $f_1 \in A^{n_1}(\Sigma)$ under the isomorphism of Section \ref{sec:MW}, and letting it act on $w_2 \in \MW_{n-n_2}$ to obtain $w_1w_2 = f_1\cdot w_2 \in \MW_{n-n_1-n_2}$. 
\medskip

\noindent \textbf{\textsf{Computing degrees.} }
Since $A$ is generated in degree $1$, we can iterate \eqref{eq:cap}  to compute the product of any two Minkowski weights. In particular, this gives a method for computing $\deg(w_1 \cdots w_k)$ for any $w_i \in \MW_{n-n_i}(\Sigma)$ with $n_1+\cdots+n_k=n$.

\medskip

\noindent \textbf{\textsf{The hyperplane and reciprocal hyperplane classes.}}
From the description of $\alpha$ and $\beta$ as Minkowski weights in $\MW_{n-1}(\Sigma_E) \cong \Hom(A^{n-1}(\Sigma_E),\RR)$, we get representations of $\alpha$ and $\beta$ in $\Hom(A^{n-1}(\Sigma_E),\RR)$ as
\begin{eqnarray*}
\alpha (x_{\F}) &=& \begin{cases}
1 & \text{ if } \F=\{i_0 \subset i_0i_1 \subset \cdots \subset i_0i_1\ldots i_{n-2}\} \\
0 & \text{ otherwise, }
\end{cases} \\
\beta (x_{\F}) &=& \begin{cases}
1 & \text{ if } \F=\{i_0i_1 \subset \cdots \subset i_0i_1\ldots i_{n-2} \subset i_0i_1\ldots i_{n-2}i_{n-1}\} \\
0 & \text{ otherwise. }
\end{cases}
\end{eqnarray*}
We invite the reader to check that these are precisely the results of multiplying $x_{\F}$ by $\alpha, \beta \in A^1(\Sigma_E)$, as described in Section \ref{sec:A}.

\begin{example}
(The degree of $\alpha\beta$ in $\Sigma_{[0,2]}$.)
Let's regard $\alpha$ as  a piecewise linear function and $\beta$ as a Minkowski weight:

\begin{figure}[H]
\centering
\scalebox{.4}
{
\begin{tikzpicture}[root/.style={circle,draw,inner sep=1.2pt,fill=black},every node/.style={scale=1.5}]
\draw[style=thin,color=gray] (2,0) -- (-2,0);
\draw[style=thin,color=gray] (-1,-1.72) -- (1,1.72);
\draw[style=thin,color=gray] (1,-1.72) -- (-1,1.72);
\node at (-1.9,0) [label=left:{$\alpha_0 = $}] {};
\node at (0,-1.5) {$t_{01}$};
\node at (0,1.5) {$t_{02}$};
\node at (1.3,-0.7) {$t_{01}$};
\node at (-1.3,-0.7) {$0$};
\node at (1.3,0.7) {$t_{02}$};
\node at (-1.3,0.7) {$0$};
\node at (1.9,0) [label=right:{ $ = \min_i(t_0-t_i) \in A^1(\Sigma),$}] {};
\end{tikzpicture}
}
\qquad
\scalebox{.4}
{
\begin{tikzpicture}[root/.style={circle,draw,inner sep=1.2pt,fill=black},every node/.style={scale=1.5}]
\node at (-1.9,0) [label=left:{$\beta= \,\,\, \e_{12}$}] {};
\draw[style=ultra thick,color=gray] (-2,0) -- (0,0);
\draw[style=ultra thick,color=gray] (1,-1.72) -- (0,0);
\draw[style=ultra thick,color=gray] (1,1.72) -- (0,0);
\node at (1.5,0) [label=right:{ $ \in \MW_1(\Sigma).$}] {};
\node at (.8,-0.7) [label={[label distance=5pt]300:{$\e_{02}$}}] {};
\node at (.8,0.7) [label={[label distance=5pt]60:{$\e_{01}$}}] {};
\end{tikzpicture}
}
\end{figure}
\noindent
Then $\alpha \cdot \beta$ is a $0$-dimensional Minkowski weight, whose weight at the origin $\bullet$  is
\begin{eqnarray*}
(\alpha \cdot \beta)(\bullet) &=& \alpha(\e_{12}) + \alpha(\e_{01}) + \alpha(\e_{02}) - \alpha(\e_{12}+\e_{01}+\e_{02}) \\
&=& 1 + 1 + 0 - 0 = 2,
\end{eqnarray*}
so the degree of $\alpha \beta$ is $2$.

\end{example}

\subsection{Morphisms} \label{sec:morphisms}

A \defword{morphism} from a fan $\Sigma$ in $\N=\mathbb{R} \otimes \N_{\mathbb{Z}}$ to a fan $\Sigma'$ in $\N' =\mathbb{R} \otimes \N'_{\mathbb{Z}}$ is an integral linear map from $\N$ to $\N'$ such that the image of any cone in $\Sigma$ is a subset of a cone in $\Sigma'$. In the context of toric geometry, a morphism from $\Sigma$ to $\Sigma'$ can be identified with a toric morphism from the toric variety of $\Sigma$ to the toric variety of $\Sigma'$  \cite[Chapter 3]{CoxLittleSchenck}.

Let  $f\colon\Sigma \to \Sigma'$ be a morphism of simplicial fans.
The pullback of functions defines the \defword{pullback homomorphism} between the Chow rings
\[
f^* \colon A(\Sigma') \longrightarrow A(\Sigma),
\]
whose dual is the \defword{pushforward homomorphism} of Minkowski weights
\[
f_* \colon \MW(\Sigma) \longrightarrow \MW(\Sigma').
\]
Since $f^*$ is a homomorphism of graded rings, $f_*$ is a homomorphism of graded modules.
In other words, the pullback and the pushforward homomorphisms satisfy the \defword{projection formula}
\[
\eta \cap f_* w = f_*(f^* \eta \cap w).
\]
for any $\eta \in A(\Sigma')$ and $w \in \MW(\Sigma)$.

\subsection{Geometry: the cohomology and Chow ring of a toric variety}

%

\noindent \textbf{\textsf{The Chow ring.}}
When $\Sigma$ is complete and simplicial, the ring $A(\Sigma)$ is both the cohomology ring and the Chow ring of the toric variety $X_\Sigma$ \cite{BrionEquivariant, CoxLittleSchenck, FultonToric}\footnote{In \cite{BrionEquivariant}, Brion identifies $A(\Sigma)$ with the Chow group of $X_\Sigma$ with real coefficients. For the existence of the ring structure and the pullback, see \cite{Vistoli}.}:
\[
H^\bullet(X_{\Sigma},\RR) \cong A(X_\Sigma) \cong A(\Sigma)
\] 
The analogous isomorphism also holds over $\ZZ$ when $\Sigma$ is unimodular \cite{Danilov}.

\medskip

\noindent \textbf{\textsf{The dictionary.}}
Under this isomorphism, the class of the torus orbit closure of a cone $\sigma$ in $\Sigma$
is identified with $\mult(\sigma) \, x_\sigma$,
where $x_\sigma$ is the monomial $\prod_{\rho \subseteq \sigma} x_\rho$ and $\mult(\sigma)$ is the index of the sublattice
$
( \sum_{\rho \subseteq \sigma} \mathbb{Z} \e_{\rho})$ in the lattice $\N_\mathbb{Z} \cap ( \sum_{\rho \subseteq \sigma} \mathbb{R} \e_{\rho} ).
$
All the fans appearing in this paper will be unimodular, so $\mult(\sigma)=1$ for every  $\sigma$ in $\Sigma$.

\medskip

\noindent \textbf{\textsf{Computing degrees.} }
In the Chow ring of a general algebraic variety, computing degrees is a rich and subtle problem for which intersection theory provides a powerful toolkit; see for example \cite{Fulton}. In the special case of toric varieties, the previous sections provide several useful methods.

\medskip

\noindent \textbf{\textsf{The hyperplane and reciprocal hyperplane classes.}}
The braid fan $\Sigma_E$ refines the normal fans $\Delta_E$ and $-\Delta_E$ of the standard and inverted simplices $\conv\{\e_i \, : \, i \in E\}$ and $\conv\{-\e_i \, : \, i \in E\}$. This gives morphisms of toric varieties $\pi_1:X_{\Sigma_E} \rightarrow X_{\Delta_E} \cong \mathbb{P}^E$ and  $\pi_2:X_{\Sigma_E} \rightarrow X_{-\Delta_E} \cong \mathbb{P}^E$, where the two copies of $\mathbb{P}^E$ are related to each other by the Cremona transformation $\mathbb{P}^E \dashrightarrow \mathbb{P}^E$ given by $(z_i)_{i \in E} \mapsto (z'_i)_{i \in E}$ where $z'_i = z_i^{-1}$.
The classes 
\[
\alpha_i = \pi_1^*(z_i=0), \qquad \beta_i=\pi_2^*(z'_i=0)
\]
in the Chow ring $A(X_{\Sigma_E})$ are the pullbacks of the hyperplane classes $z_i=0$ and $z'_i=0$ in the respective copies of $A({\mathbb{P}^E})$.

\begin{example}
(The degree of $\alpha\beta$ in $\Sigma_{[0,2]}$.) 
Let's compute the degree of $\alpha \beta$ from first principles. Away from the coordinate subspaces, we need to compute the number of intersections of a generic hyperplane $\alpha$ and a generic reciprocal hyperplane $\beta$:
\[
z_0=az_1+bz_2, \qquad 
\frac1{z_0}=\frac{c}{z_1} + \frac{d}{z_2}.
\]
Setting $z=z_1/z_2$, this system is equivalent to the equation $adz^2 + (ac+bd)z + bc=0$, which has two solutions for generic $a,b,c,d$.
Therefore $\deg(\alpha \beta) = 2$.
\end{example}

\section{Intersection theory of matroids: four approaches} \label{sec:matroid}

\setcounter{subsection}{-1}

\subsection{Matroids, characteristic polynomials, and matroid fans}\label{subsec:matroid}

Let us introduce some basic definitions on matroids, and discuss three combinatorial ways to compute the coefficients $\mu^0, \ldots, \mu^r$ of the reduced characteristic polynomial of $\M$.

\bigskip

A \defword{matroid} $M=(E, r)$ consists of a finite set $E$ and a function $r: 2^E \rightarrow \ZZ$, called the \defword{rank function} such that \\
\smallskip
\noindent (R1) $0 \leq r(A) \leq |A|$ for all $A \subseteq E$, \\
\smallskip
\noindent (R2) $r(A) \leq r(B)$ for all $A \subseteq B \subseteq E$, and \\
\smallskip
\noindent (R3) $r(A)+r(B) \geq r(A\cup B) + r(A \cap B)$ for all $A,B \subseteq E$.

A motivating example is the matroid of a vector configuration $E \subset \FF^d$, whose rank function is given by
\[
r(A) = \dim(\span \, A) \qquad \text{ for } A \subseteq E.
\]
Such a matroid is said to be \defword{linear over} $\FF$.

\bigskip

\noindent
\textbf{\textsf{The lattice of flats.}}
A \defword{flat} of $\M$ is a subset $F \subseteq E$ such that 
$
r(F \cup e) > r(F) \textrm{ for all } e \notin F.
$
We say a flat $F$ is \defword{proper} if it does not have rank $0$ or $r$.
The \defword{lattice of flats}  $L_{\M}$ is the set of flats, partially ordered by inclusion. Its minimum and maximum element are called $\widehat{0}$ and $\widehat{1}$, and its least upper bound and greatest lower bound maps are denoted $\wedge$ and $\vee$.
When $\M$ is the matroid of a vector configuration $E$ in a vector space $V$,  the flats of $\M$ correspond to the subspaces of $V$ spanned by subsets of $E$, as illustrated in Figure \ref{fig:matroid}.

\begin{figure}[h]
 \begin{center}
    \includegraphics[height=2cm]{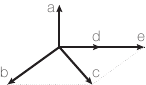} \qquad \quad 
  \includegraphics[height=3cm]{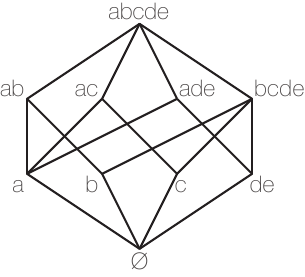} \qquad \quad 
    \includegraphics[height=3.5cm]{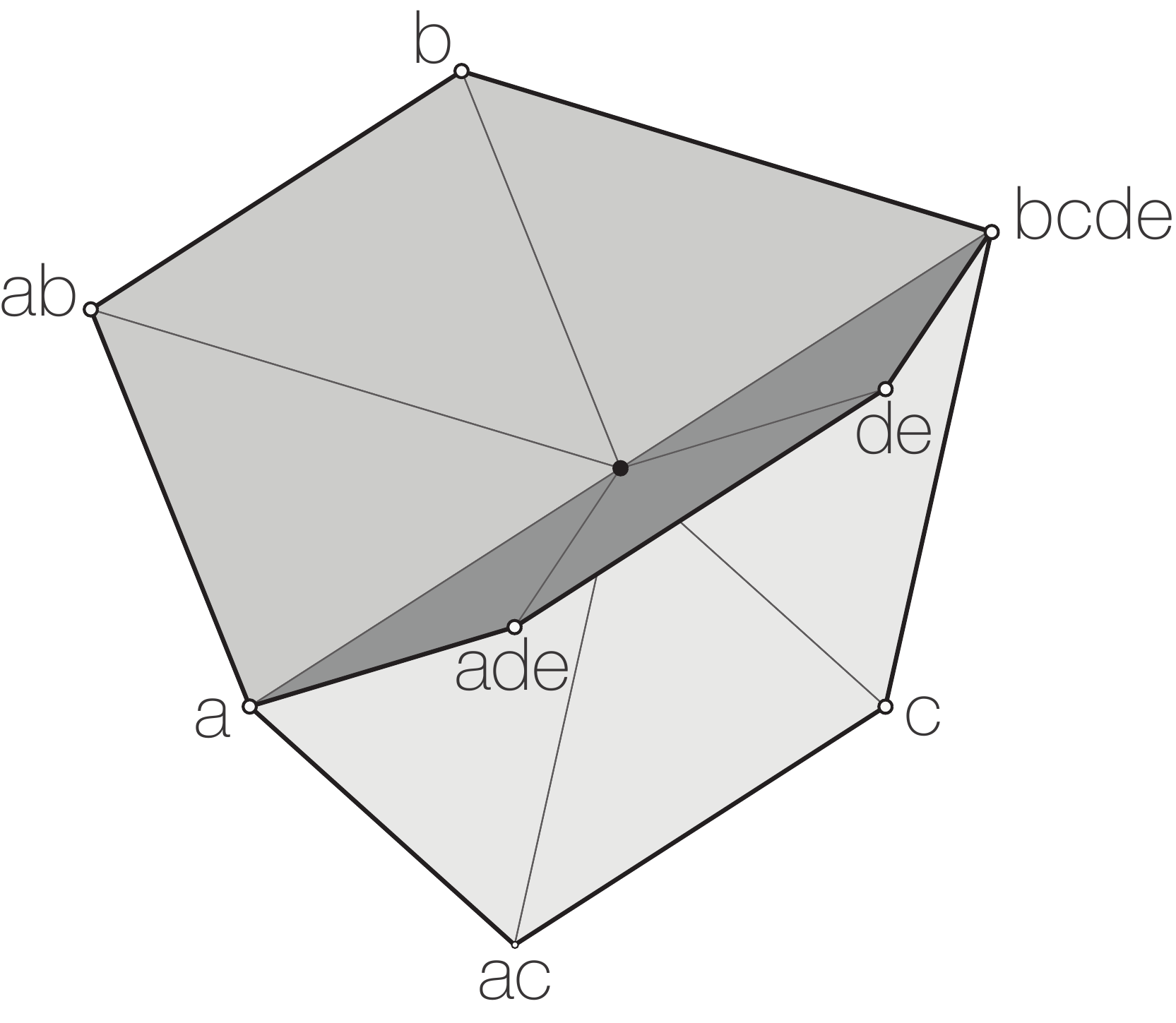}
  \caption{  \label{fig:matroid} A vector configuration, its lattice of flats, and its matroid fan.}
 \end{center}
\end{figure}

The \defword{M\"obius function} of $L_{\M}$ is the function $\mu: L_{\M} \rightarrow {\mathbb{Z}}$ defined by
\begin{equation}\label{f.e:Mobius1}
\sum_{G \leq F} \mu(G) = \begin{cases}
1 & \textrm{ if } F = \widehat{0}, \\
0 & \textrm{ otherwise}.
\end{cases}
\end{equation}
The \defword{M\"obius number} of $\M$ is $\mu(\M) \coloneqq \mu(\widehat{1})$.
The characteristic polynomial, which was defined in terms of the rank function in Section \ref{sec:intro},  can be expressed in terms of the M\"obius function:
\begin{equation}\label{eq:chimu}
\chi_{\M}(q) = \sum_{F \in L_{\M}} \mu(F) q^{r(\M)-r(F)}.
\end{equation}
Whitney's theorem gives the alternative expression in \eqref{eq:whitney}.

\medskip

\noindent
\textbf{\textsf{Matroid constructions and three properties of $\mu^k$.}}
 Let $e$ be an element of $E$.
The \defword{deletion} $\M \backslash e$ and \defword{contraction} $\M/e$ are the matroids on $E-e$ with rank functions
\[
r_{\M \backslash e}(A) = r_{\M}(A) \quad \text{and} \quad 
r_{\M / e}(A) = r_{\M}(A\cup e) - r_{\M}(e) \quad \text{ for } A \subseteq E-e.
\]
If $\M$ is the matroid of a vector configuration $E \subseteq \FF^d$ 
then $\M \backslash e$ and $M/e$ are the matroids of the vector configuration $E-e \subseteq \FF^d$ and its image $\overline{E-e}$ in the quotient vector space $\FF^d/\FF e$. It follows from the definition in Section \ref{sec:intro} that the characteristic polynomial satisfies the \defword{deletion-contraction} recurrence $\chi_{\M}(q) = \chi_{\M \backslash e}(q) - \chi_{\M/e}(q)$, which gives
\begin{equation}\label{eq:mudelcont}
\mu^k(\M) = \mu^{k-1}(\M/e) + \mu^{k}(\M \backslash e) \qquad \text{ for } 0 \leq k \leq r.
\end{equation}

The \defword{truncation} $\Tr \M$ is the rank $r$ matroid obtained from $\M$ by omitting the flats of rank $r$. If $\M$ is the matroid of a vector configuration $E \subseteq \FF^d$ over a field of characteristic $0$, then $\Tr \M$ is the matroid of the projection of $E$ onto a generic hyperplane $H$ of $\FF^d$. It follows from \eqref{eq:chimu} that the first $r-1$ coefficients of $\chi_{\M}(q)$ and $\chi_{\Tr \M}(q)$ match; a simple calculation then gives
\begin{equation}\label{eq:mutrunc}
\mu^k(\M) = (-1)^k \mu(\Tr^{r+1-k}\M) 
 \qquad \text{ for } 0 \leq k \leq r.
\end{equation}

Let's label each edge from $F$ to $G$ in the Hasse diagram of $L_{\M}$ with the element $\min_<(G-F)$. The \defword{Jordan-H\"older sequence} $\pi(\m)$ of a maximal chain $\m$ from $\widehat{0}$ to $\widehat{1}$ is the sequence of labels from the bottom to the top. Its \defword{descent set} records the positions where this sequence decreases: $D(\m)=\{i \in [r] \, : \, \pi(\m)_i > \pi(\m)_{i+1}\}$.  
Stanley proved \cite[Theorem 2.7]{Bjorner} that the number of maximal chains $\m$ whose Jordan-H\"older sequence $\pi(\m)$ has descent set $D(\m) = S$
equals the M\"obius number $(-1)^{|S|+1} \mu((L_{\M})_{S})$ of the rank-selected subposet $(L_{\M})_{S} = \{F \in L_{\M} \, : \, r(F) \in S\}$. In particular, if $S=[k]$ then $(L_{\M})_{[k]}$ is the lattice of flats of ${\Tr^{r+1-k} \M}$. Therefore
\begin{equation} \label{eq:muD}
\mu^k(\M) = \text{\# of maximal chains $\m$ in $L_{\M}$ with descent set } D(\m) = [k].
\end{equation}
This edge labeling is important in the study of the topology of $L_{\M}$; see \cite{Bjorner}.

\medskip

\noindent
\textbf{\textsf{Matroid fans.}}
 Sturmfels \cite{Sturmfels} and Ardila and Klivans \cite{ArdilaKlivans} introduced the \defword{matroid fan} or \defword{Bergman fan} of a matroid:

\begin{Def} \cite{ArdilaKlivans} Let $\M$ be a matroid on ground set $E$.
The \defword{matroid fan} or \defword{Bergman fan} $\Sigma_{\M}$ in $\N_E=\RR^E/\RR\e_E$ has
\begin{itemize}
\item rays: $\displaystyle \e_F$ for the proper flats $\emptyset \subsetneq F \subsetneq E$, and
\item cones: $\sigma_{\F} = \cone(\e_{F_1}, \ldots, \e_{F_k})$ for the flags $\F=(\varnothing \subsetneq F_1 \subsetneq \cdots \subsetneq F_k \subsetneq E)$. 
 \end{itemize}
\end{Def}

If $\M$ has rank $r+1$, the braid fan $\Sigma_{\M}$ is a pure $r$-dimensional subfan of the braid fan $\Sigma_E$. Notice that $\Sigma_E$ is the matroid fan for the \defword{Boolean matroid} where $r(A) = |A|$ for all $A \subseteq E$.

\begin{prop} \label{prop:matroidbalanced}
The matroid fan is balanced with unit weights. 
\end{prop}

\begin{proof} Consider any $(r-1)$-face of the braid fan; we can write it as $\tau = \sigma_{\F}$ for $\F=(\varnothing \subsetneq F_1 \subsetneq \cdots F_{i-1}  \subsetneq F_{i+1} \subsetneq \cdots \subsetneq F_r \subsetneq E)$ where $r(F_j)=j$ for all $j$. The  facets of $\Sigma_E$ containing $\tau$ are those of the form $\sigma = \sigma_{\F \cup F}$ for the rank $i$ flats $F$ with $F_{i-1}  \subsetneq F \subsetneq F_{i+1}$.
These correspond to the lines $F-F_{i-1}$ of the rank $2$ matroid $\M[F_{i-1},F_{i+1}] = (M \backslash (E-F_{i+1}))/F_i$, whose union is its ground set $F_{i+1}-F_{i-1}$. Therefore
\[
\sum_{\tau \subset \sigma} w(\sigma)\e_{\sigma/\tau} = 
\sum_{F_{i-1}  \subsetneq F \subsetneq F_{i+1}} \e_{F-F_{i-1}} = 
\e_{F_{i+1}-F_{i-1}} = 0 \text{ in } \N/\span{\tau}.
\]
as desired.
\end{proof}

It follows that we can regard $\Sigma_{\M}$, with unit weights, as a Minkowski weight $1_{\M}$ on the matroid fan $\Sigma_{\M}$ or on the permutahedral fan $\Sigma_E$.


The \defword{Chow ring} $A(\M)$ of $\M$ is the Chow ring $A(\Sigma_{\M})$ of its matroid fan, as defined in Section \ref{sec:A}. Even though $\Sigma_{\M}$ is not complete, $A(\M)$ also has a degree map \cite{AdiprasitoHuhKatz}:
\begin{eqnarray*}
\deg_{\M}: A^r(\M) & \rightarrow & \RR \\
\eta & \mapsto & \eta \cdot 1_{\M}.
\end{eqnarray*}

%

\medskip

\noindent
\textbf{\textsf{The theme.}}
The inclusion $i:\Sigma_{\M} \rightarrow \Sigma_E$ of the matroid fan in the braid fan  is a morphism of fans.
As explained in Section \ref{sec:morphisms}, this gives pullback and pushforward homomorphisms
\[
i^* \colon A(\Sigma_E) \longrightarrow A(\Sigma_{\M}), \qquad i_* \colon \MW(\Sigma_{\M}) \longrightarrow \MW(\Sigma_E)
\]
satisfying the \defword{projection formula} $\eta \cdot i_* w = i_*(i^* \eta \cdot w)$.

The classes $\alpha_E$ and $\beta_E$ of the braid Chow ring $A(\Sigma_E)$ described in Section \ref{sec:toric}
 pull back to the \defword{hyperplane} and \defword{reciprocal hyperplane classes} 
\[
\alpha_{\M} \coloneqq i^*(\alpha_E), \qquad \beta_{\M} \coloneqq i^*(\beta_E)
\]
of the matroid Chow ring $A(\Sigma_{\M})$. Also, 
 the top-dimensional constant Minkowski weight $1_{\M}$ on the matroid fan $\Sigma_{\M}$ pushes forward to the Minkowski weight $i_*(1_{\M}) = \Sigma_{\M}$ on the braid fan $\Sigma_E$. 
The projection formula then gives
\[
\deg_{{\M}}(\alpha_{\M}^{r-k}\beta_{\M}^k) =  \deg_{E} (\Sigma_{\M} \cdot \alpha_E^{r-k}\beta_E^k).
\]
where $\deg_{\M}: A^r(\Sigma_{\M}) \xrightarrow{\sim} \mathbb{R}$ and
$\deg_{E}: A^n(\Sigma_E) \xrightarrow{\sim} \mathbb{R}$ 
are the degree map of $\Sigma_{\M}$ and $\Sigma_E$, respectively. Now we restate our main theme:

\bigskip

\framebox{
\begin{minipage}{12.5cm}

\textbf{Theorem 1.1} \defword{
Let $\M$ be a matroid  of rank $r+1$. Let $\alpha, \beta$ be the  hyperplane and reciprocal hyperplane classes in the Chow ring $A(\M)$. Then}
\[
\deg_{\M}(\alpha^{r-k} \beta^k) = \mu^k(\M) \qquad \text{ for } 0 \leq k \leq r.
\]

\end{minipage}}

\bigskip

\noindent and devote the rest of the paper to four variations on its proof.


\subsection{The Chow ring as a quotient of a polynomial ring}\label{sec:matroidBrion}

Let $\M$ be a loopless matroid  of rank $r+1$ on a set $E$ with $n+1$ elements.
The Chow ring
$A(\M)$ is the $\RR$-algebra generated by variables $x_F$ for each non-empty proper flat and relations
\begin{eqnarray*}
x_Fx_G=0 && \text{ for flats $F, G$ such that $F \subsetneq G$ and $F \supsetneq G$}, \\
\sum_{F \ni i}x_F = \sum_{F \ni j} x_F && \text{ for elements $i, j \in E$}.
\end{eqnarray*}
The Chow ring is graded $A(\M) = A^0(\M) \oplus \cdots \oplus A^r(\M)$, and the isomorphism $\deg_{\M}: A^r(\M) \rightarrow \RR$ 
is characterized by its value on square-free monomials:
\[
\deg(x_{F_1}\cdots x_{F_k}) = \begin{cases}
1 & \text{ if } F_1, \ldots, F_k \text{ form a flag}, \\
0 & \text{ otherwise.}
\end{cases}
\qquad 
\text{ for } F_1, \dots, F_k \text{ distinct.}
\]

In this presentation, the \defword{hyperplane} and \defword{reciprocal hyperplane} classes $\alpha_{\M}$ and $\beta_{\M}$ are given by:
\[
\alpha = \alpha_i =  \sum_{i \in F} x_F, \qquad
\beta = \beta_i =  \sum_{i \notin F} x_F.
\]
As before, that they do not depend on $i$. 

Our goal is to compute the degree of $\alpha^{r-k} \beta^k$ in the Chow ring $A(\M)$. To do so, 
 we seek to express $\alpha^{r-k} \beta^k$ as a sum of square-free monomials, each of which have degree one.
One fundamental feature of this computation, which is simultaneously a challenge and an advantage, is that there are many ways to carry it out. 
We are free to choose any one of the $E$
different expressions for $\alpha$ and $\beta$ to compute. To have control over the computation, we require some structure amidst that freedom. Let us prescribe a precise way of carrying out these computations, in terms of a fixed linear order $<$ on the ground set $E$ of $\M$.

\begin{Def}  \label{def:lexicographic} 
Let $\F = \{\emptyset \subsetneq F_1 \subsetneq \cdots \subsetneq F_k \subsetneq E\}$ be a flag of flats of $\M$.

\noindent
$\bullet$
The \defword{lexicographic expansion of $x_{\F}\,  \alpha$} is the expression
\[
x_{\F} \,\alpha
=
x_{\F} \,\alpha_{e}
=
 \sum_{
F \supset F_k \cup e
} x_{\F} x_{F},
\]
where $e=\min_<(E-F_k)$ is the $<$-smallest element of $E$ that \textbf{is not} in $F_k$. Note that since $e \in F$ and $e \notin F_k$, the new flat $F$ in each term must be the maximal flat in the new flag $\F \cup F$.

\noindent
$\bullet$
The \defword{lexicographic expansion of $x_{\F}\,  \beta$} is the expression
\[
x_{\F} \,\beta
=
x_{\F} \,\beta_{e}
=
 \sum_{
F \subset F_1 - e
} x_F x_{\F} ,
\]
where $e=\min_< F_1$ is the $<$-smallest element of $E$ that \textbf{is} in $F_1$.
Note that since $e \notin F$ and $e \in F_1$, the new flat $F$ must be the minimal flat in the new flag $F \cup \F$. 

\noindent
$\bullet$
The \defword{lexicographic expansion of $x_{\F}\,  \beta^t$} is obtained recursively by multiplying each monomial in the lexicographic expansion of $x_{\F}\,  \beta^{t-1}$ by $\beta$, again using the lexicographic expansion.

\noindent
$\bullet$
The \defword{lexicographic expansion of $x_{\F}\,  \alpha^s \beta^t$} is obtained recursively by multiplying each monomial in the lexicographic expansion of $x_{\F}\,  \alpha^{s-1}\beta^t$ by $\alpha$, again using the lexicographic expansion.
\end{Def}

By construction, these lexicographic expansions are sums of non-zero square-free monomials in $A({\M})$.
We invite the reader to compute the lexicographic expansions of $\alpha^2, \alpha\beta$, and $\beta^2$ for the matroid in Figure \ref{fig:matroid}.
We now describe the outcome of this computation in general.

A flag $\F = \{\emptyset \subsetneq F_1 \subsetneq \cdots \subsetneq F_k \subsetneq E\}$ gives rise to a word 
\[
\m(\F) = m_1m_2 \ldots m_{k+1} \qquad \textrm{ where } m_i = \min_<(F_i - F_{i-1})
\]
and a \defword{descent set}
\[
D(\F) = \{i \in [k] \, : \, m_i > m_{i+1}\}.
\]

\begin{lemma}
The lexicographic expansion of $\alpha^s \beta^t$ is
\[
\alpha^s \beta^t = \sum_{\F \, : \, |\F| = s+t, \, D(\F) = [t] } x_{\F}.
\]
\end{lemma}

\begin{proof}
In the terms $x_{F \cup \F}$ of the lexicographic expansion of $x_{\F} \beta$, the condition $F \subset F_1 - e$ is equivalent to $\min F > \min (F_1-F) = e$, that is, to an initial descent in the word of $F \cup \F$.

In the terms $x_{\F \cup F}$ of the lexicographic expansion of $x_{\F} \alpha$, the condition $F \supset F_k \cup e$ is equivalent to $\min (F-F_k) = e > \min (E-F)$, that is, to a final ascent in the word of $\F \cup F$. 

Since the lexicographic expansion in question is computed recursively in the order $1, \beta, \beta^2, \ldots, \beta^t, \alpha \beta^t, \alpha^2\beta^t, \ldots, \alpha^s\beta^t$, its terms correspond to the flags of length $s+t$ whose words have $t$ initial descents and $s$ final ascents.
\end{proof}

\begin{proof} \textbf{ \hspace{-.4cm} 1 of Theorem \ref{thm:main}:}
The flags of length $(r-k)+k$ whose words have $k$ initial descents and $r-k$ final ascents correspond to the maximal chains in the lattice $L_{\M}$ whose Jordan-H\"older sequence has descent set $[k]$. As we discussed in Section \ref{subsec:matroid}, there are $\mu^k$ such flags.
\end{proof}

This first proof of Theorem 1.1 is based on \cite{AdiprasitoHuhKatz}.
For an alternative deletion-contraction proof motivated by the intersection theory of moduli spaces of curves, see \cite{DastidarRoss}.

\subsection{The Chow ring in terms of piecewise polynomials}\label{sec:matroidPP}

For a permutation $\sigma:[0,n] \rightarrow E$, let $B_\sigma$ be the lexicographically smallest basis of $\M$ with respect to the order $\sigma$. It  can be constructed greedily, by sequentially adding each of $\sigma(0), \sigma(1), \ldots, \sigma(n)$ as long as it is independent from the previously added elements.

The following piecewise polynomial functions are representatives for the classes of  $\Sigma_{\M}$, $\alpha$, and $\beta$ in $\PP(\Sigma_E)/N^\vee$. For 
any fixed element $f \in E$, 
\begin{equation} \label{eq:PP}
[\alpha]_\sigma = t_f-t_{\sigma(n)}, \quad 
[\beta]_\sigma = t_{\sigma(0)}-t_f, \quad
[\Sigma_{\M}]_\sigma = (-\t)_{E-B_\sigma} \coloneqq \prod_{i \notin B_\sigma} (t_f-t_i),
\end{equation}
for each permutation $\sigma$ of $E$:
The first two equalities were shown in Section \ref{sec:PP}; for the third, see \cite[Lemma 4.3]{ArdilaEurPenaguiao} and \cite[Theorem 7.6]{BEST}.
Therefore, as explained in Section \ref{sec:PP},
\[
\deg_E(\alpha^{r-k} \beta^k \Sigma_{\M}) = 
\sum_{\sigma \in S_E}
\frac{(t_{\sigma(0)}-t_f)^{k} (t_f-t_{\sigma(n)})^{r-k} \prod_{i \notin B_\sigma} (t_f-t_i)}
{(t_{\sigma(0)} - t_{\sigma(1)})(t_{\sigma(1)} - t_{\sigma(2)}) \cdots (t_{\sigma(n-1)} - t_{\sigma(n)})}
\]
and we need to prove that this rational function equals the constant $\mu^k(\M)$. 

Let us write $\t_\sigma\coloneqq {(t_{\sigma(0)} - t_{\sigma(1)})(t_{\sigma(1)} - t_{\sigma(2)}) \cdots (t_{\sigma(n-1)} - t_{\sigma(n)})}$ for each bijection $\sigma:[0,n]\rightarrow E$, written in ``one-line notation" as the word $\sigma(0)\ldots \sigma(n)$.
Also recall that $t_{ij}\coloneqq t_i-t_j$. 
Finally write $[m_k(\M)] \coloneqq [\alpha^{r-k} \beta^k \Sigma_{\M}]$ regarded as a piecewise polynomial function on $\Sigma_E$, so that
\[
\deg_E(\alpha^{r-k} \beta^k \Sigma_{\M}) = 
\sum_{\sigma \in S_E} \frac{[m_k(\M)]_\sigma}{\t_\sigma}.
\]

\begin{proof} \textbf{ \hspace{-.4cm} 2 of Theorem \ref{thm:main}:}
We will prove that this sum equals $\mu^k(\M)$ by showing
that it satisfies a deletion-contraction recurrence. Let $e$ be an element that is neither a loop nor a coloop.\footnote{A similar analysis, which we omit, will hold when $e$ is a loop or a coloop.}
Each permutation of $E$ can be written uniquely in the form $\tau^i =  \tau(0) \ldots \tau(i-1) \ e \ \tau(i) \ldots \tau(n-1)$
for a permutation $\tau=\tau(0) \ldots \tau(n-1)$ of $E-e$ and an index $0 \leq i \leq n$.
For each permutation $\tau$ of ${E-e}$, there is an element $j \in E-e$ such that 
\[
B_\tau(\M/e) =: B_\tau, \qquad B_\tau(\M \backslash e) = B_\tau \cup \tau(j).
\]
Then we have
\[
B_{\tau^i}(\M) = 
\begin{cases} 
B_\tau \cup e  & \text{ if } i \leq j, \\ 
B_\tau \cup \tau(j)  & \text{ if } i > j. 
\end{cases}
\]

Now we use this to compute the parts of a piecewise polynomial function representing $[m_k(\M)] \coloneqq [\alpha^{r-k} \beta^k \Sigma_{\M}]$ recursively. We use the equations in \eqref{eq:PP} which are valid for any $f \in E$; we will choose $f=e$.\footnote{One can prove the recurrence without making this choice, but that requires additional ideas.}
Notice that 
$[\beta]_{\tau^0}=0$,  
$[\alpha]_{\tau^n}=0$,
and  $[\Sigma_M]_{\tau^i} = 0$ for $i>j$ since $f=e$ and $e \notin B_{\tau^i}.$ 
Therefore
\[
[m_k(\M)]_{\tau^i} = \begin{cases}
(t_{\tau(0)e})^k(t_{e\tau(n-1)})^{r-k} (-\t)_{E-B-e} &  \text{ for } i=1, \ldots, j, \\
0  &  \text{ for } i=0, j+1, \ldots, n-1, n .
\end{cases}
\]

Let us sum the contributions from permutations $\tau^0, \ldots, \tau^n$ to $\deg [m_k(\M)]$:
\begin{eqnarray*}
&& \sum_{i=0}^n 
\frac{[m_k(\M)]_{\tau^i}}{\t_{\tau^i}} \\
&=& \frac{(t_{\tau(0)e})^k(t_{e\tau(n-1)})^{r-k} (-\t)_{E-B-e}}{\t_\tau} 
\left[\sum_{i=1}^j \frac{\t_\tau}{\t_{\tau^i}}\right] \\
&=& \frac{(t_{\tau(0)e})^k(t_{e\tau(n-1)})^{r-k} (-\t)_{E-B-e}}{\t_\tau} 
 \cdot \left[ \frac1{t_{\tau(0)e}} + \frac1{t_{e\tau(j)}}\right]\\
&=& \frac{(t_{\tau(0)e})^{k-1}(t_{e\tau(n-1)})^{r-k} (-\t)_{(E-e)-B}}{\t_\tau} 
+ \frac{(t_{\tau(0)e})^k(t_{e\tau(n-1)})^{r-k} (-\t)_{(E-e)-B - \tau(j)}}{\t_\tau} \\
&=&\frac{[m_{k-1}(\M / e)]_{\tau}}{\t_{\tau}} + \frac{[m_k(\M \backslash e)]_{\tau}}{\t_{\tau}},
\end{eqnarray*}
using the fact that the sum of $\displaystyle \frac{\t_\tau}{\t_{\tau^i}} = 
\frac{t_{\tau(i-1)\tau(i)}}{(t_{\tau(i-1)e})(t_{e\tau(i)})} = \frac1{t_{\tau(i-1)e}} - \frac1{t_{\tau(i)e}}$ telescopes. 

Summing this equality over all $\tau \in S_{E-e}$, we obtain
\begin{eqnarray*}
\deg_E(m_k(\M)) &=& \deg_E(m_{k-1}(\M / e))  + \deg_E(m_k(\M \backslash e)) \\
&=& \mu^{k-1}(\M/e) + \mu^k(\M \backslash e) \\ 
&=& \mu^k(\M)
\end{eqnarray*}
as desired.
\end{proof}

For a similar recursive proof of a much more general statement, see \cite{BEST}.

\subsection{The Chow ring in terms of Minkowski weights}

Let us now compute the degree of $\alpha^{r-k}\beta^k\Sigma_{\M}$  stable intersection of the respective Minkowski weights. We have already described the matroid fan $\Sigma_{\M}$.
The tropical fans of $\alpha^{r-k}$ and $\beta^k$ are
the subfans of $\Sigma_E$ with unit weights and support $\Sigma_{E,r-k}$ and 
$-\Sigma_{E,k}$ where 
\[
\Sigma_{E,i} = \{\x \in \N_E \, : \, \textrm{the smallest $i+1$ coordinates of $\x$ are equal}\}.
\]
For each $S \subseteq E$ we consider the cone of $\N_E$ where the $S$ coordinates are minimal:
\[
\Sigma_{E,S} = \{\x \in \N_E \, : \, x_s \leq x_t \text{ for all } s \in S, t \in E\}.
\]
The relative interiors are pairwise disjoint, and $\displaystyle \Sigma_{E,i} = \bigcup_{|I|=i+1} \Sigma_{E,I}$ for $1 \leq i \leq n$.

\begin{proof} \textbf{ \hspace{-.4cm} 3 of Theorem \ref{thm:main}:}
Let $\a$ and $\b$ be generic vectors with $\a$ decreasing and $\b$ increasing, so $a_0 > a_1 > \cdots > a_n$ and $b_0 < b_1 < \cdots < b_n$
We need to find the points in the intersection $\Sigma_{\M} \cap (\a + \Sigma_{E,r-k}) \cap (\b - \Sigma_{E,k})$.
Let us compute the points

\[
\x \in \sigma_{\F} \cap (\a + \Sigma_{E,I}) \cap (\b - \Sigma_{E,J}), \qquad \text{for } |\F| = r, |I|=r-k+1, \, |J|=k+1.
\]
Let the flag $\F$ induce the ordered set partition $\S = S_1|\cdots|S_{r+1}$ of $E$ with parts $S_i = F_i - F_{i-1}$ for $1 \leq i \leq r+1$.

Since the codimensions of these cones add up to $(n-r)+(r-k)+k = n = \dim \N_E$ and $\a$ and $\b$ are generic, for this intersection to be nonempty, the sets $I$ and $J$ cannot share a pair of elements with each other or with any $S_h$. Therefore the parts $S_1, \ldots, S_{r+1}$ split into three types:

\medskip

(I) $r-k$ parts containing exactly one element of $I$ and none of $J$,

(J) $k$ parts containing no element of $I$ and exactly one of $J$,

(IJ) one part containing exactly one element of $I$ and one element of $J$.

\medskip

We claim that in any part of $\S$, the minimum element is the one belonging to $I$ or $J$. To show the claim, assume $i \in I$ is in a block $S \ni i$ of $\S$. Since $\x-\a \in \Sigma_{E,I}$, \ $(x-a)_i \leq (x-a)_s$ for all $s \in S$, so $a_i = \max \{a_s \, : \, s \in S\}$. Since $\a$ is decreasing, $i = \min \{s \, : \, s \in S\}$. 
An analogous argument shows that any $j \in J$ is the smallest in its block $S \ni j$. 
In particular, the minimum element $m$ of the part of type (IJ) is the element belonging to $I$ \textbf{and} $J$.

Now let us choose the representatives of $\a, \b, \x \in \N_E=\RR^E / \RR\e_E$ that make $a_{m}=b_{m} = x_{m} = 0$. Since
$\x - \a \in \Sigma_{E, I}$ and 
$\x - \b \in -\Sigma_{E, J}$, 
the smallest coordinates of $\x-\a$ and the largest coordinates of $\x-\b$, both achieved at $m \in I \cap J$, equal $0$. It follows that
\[
x_s = 
\begin{cases}
x_i = a_{i}  & \textrm{ for } s, i \in S   \textrm{ of type (I)} \\
x_j = b_{j}  & \textrm{ for } s, j \in S   \textrm{ of type (J)} \\
x_m = 0 & \textrm{ for } s, m \in S \textrm{ of type (IJ)},
\end{cases}
\qquad 
\text{ and } a_e \leq x_e \leq b_e \textrm{ for all } e \in E.
\]

We claim that $m=0$. If  $m>0$, since $\a$ is decreasing and $\b$ is decreasing, we would have
\[
0 = a_m < a_0 \leq b_0 < b_m = 0.
\]
This implies that
\[
a_n < \cdots < a_1 < a_0 = 0 = b_0 < b_1 < \cdots < b_n.
\]
Since $\x \in \sigma_{\F}$, the parts $S$ of type (I), where $x_s = a_i$ for some $i \in I-0$, must come after the part of type (IJ) in $\S$. Analogously, the parts $S$ of type (II) must come before the part of type (IJ) in $\S$.

We conclude that the parts $S_1, \ldots, S_{k}$ are of type (J), $S_{k+1}$ is  of type (IJ), and $S_{k+2}, \ldots, S_{r+1}$ are of type $(I)$. Their respective minimum elements $j_1, \ldots, j_k \in J$, $0 \in I \cap J$, $i_{k+2}, \ldots, i_{r+1}$  have decreasing $x$ coordinates, which means that $j_1\ldots j_{k}0$ is decreasing and $0i_{k+2} \ldots i_{r+1}$ is increasing. Thus $\F$ is a maximal chain whose Jordan-H\"older sequence has descent set $[k]$. As we saw in Section \ref{sec:matroidPP}, there are $\mu^k(\M)$ such chains.

Rereading the proof, the reader will see that the point $\x$ computed above does  provide an intersection point of our tropical fans; it is not difficult to verify that it has index $1$. This completes the proof that the intersection degree is $\mu^k(\M)$.
\end{proof}

\subsection{The Chow ring in terms of tropical intersection}

Let $1 \leq r_1 \leq r_2 \leq r$. The \defword{$(r_1,r_2)$--truncation} of $\M$ is the Minkowski weight $\Sigma_{\M,[r_1,r_2]}$ on the braid fan $\Sigma_E$
whose weight on the cone $\sigma_{\F}$ corresponding to a flag $\F = \{\emptyset \subsetneq F_{r_1} \subsetneq F_{r_1+1} \subsetneq \cdots \subsetneq F_{r_2} \subsetneq E\}$ is
\[
w(\sigma_{\F}) = \begin{cases}
|\mu(F_{r_1})| & \text{ if each $F_i$ is a flat of rank $i$ in $\M$, or} \\
0 & \text{ otherwise}.
\end{cases}
\]

\begin{prop} \label{prop:truncation} \cite{Huhthesis, HuhKatz} Let $\M$ be a matroid of rank $r+1$.
For $1 \leq r_1 < r_2 \leq r+1$,
\[
\alpha \cdot \Sigma_{\M,[r_1,r_2]} = \Sigma_{\M,[r_1,r_2-1]}, \qquad 
\beta \cdot \Sigma_{\M,[r_1,r_2]} = \Sigma_{\M,[r_1+1,r_2]}.
\]
For $1 \leq k \leq r$,
\[
\deg (\alpha \cdot \Sigma_{\M,[k,k]}) =  \mu^{k-1}, \qquad 
\deg (\beta \cdot \Sigma_{\M,[k,k]}) =  \mu^{k}.
\]
In particular, each $(r_1,r_2)$-truncation $\Sigma_{\M,[r_1,r_2]}$ is a balanced fan.
\end{prop}

\begin{proof}
Let us prove the statements for $\alpha$; the proofs for $\beta$ are similar and can be found in \cite{Huhthesis}. To show that $\alpha \cdot \Sigma_{\M,[r_1,r_2]} = \Sigma_{\M,[r_1,r_2-1]}$, we show that the Minkowski weight $\alpha \cdot \Sigma_{\M,[r_1,r_2]}$ has the correct value on each cone of $\Sigma_{\M,[r_1,r_2-1]}$, and has weight 0 on every other cone.

First, consider a cone $\sigma_{\G}$ in $\Sigma_{\M,[r_1,r_2-1]}$ 
with $\G = \{G_{r_1} \subsetneq \cdots \subsetneq G_{r_2-1}\}$,
and let $F_1, \ldots, F_m$ be the rank $r_2$ flats of $\M$ compatible with $\G$; that is, the flats covering $G_{r_2-1} = G$ in the lattice $L_{\M}$. They correspond to the rank $1$ flats $F_1-G, \ldots, F_m-G$ of  $\M[G,E] = \M/G$, which partition its ground set $E-G$, so 
\[
\e_{F_1} + \cdots + \e_{F_m} = (m-1) \e_{G} + \e_E = (m-1) \e_{G} \qquad \text{ in } \N_E.
\]
The full-dimensional cones of $\Sigma_{\M,[r_1,r_2]}$ containing $\sigma_{\G}$ are $\sigma_{\G \cup F_i}$ for $1 \leq i \leq m$, all with weight $|\mu(G_{r_1})|$. Let us choose an element $e \in G$ \footnote{One will obtain the same answer using $\alpha = \alpha_f$ for any other $f \in E$, but choosing $\e \in G$ simplifies the computation.} and compute the weight of the divisor $\alpha \cdot \Sigma_{\M,[r_1,r_2]}$ on this cone $\sigma_{\G}$, using $\alpha = \alpha_e$:
\begin{eqnarray*}
\alpha \cdot \Sigma_{\M,[r_1,r_2]}(\sigma_{\G}) &\coloneqq& 
\sum_{\F \supset \G} w(\sigma_{\F}) \alpha(\e_{\F/\G}) - \alpha \left(\sum_{\F \supset \G} w(\sigma_{\F}) \e_{\F/\G} \right) \\
&=& {|\mu(G_{r_1})|} \left(  \sum_{j=1}^m \alpha(\e_{F_j}) - \alpha \left(\sum_{j=1}^m \e_{F_j} \right) \right)\\
&=& {|\mu(G_{r_1})|}  \left(m - \alpha \left((m-1)\e_G \right) \right)\\
&=& {|\mu(G_{r_1})|} ,
\end{eqnarray*}
noting that $e \in G \subset F_j$ imply $\alpha(\e_G) = \alpha(\e_{F_j}) = 1$ for all $j$.

Now consider a cone $\sigma_{\G}$ of $\Sigma_{\M,[r_1,r_2]}$ that is \textbf{not} in $\Sigma_{\M,[r_1,r_2-1]}$; let the flats in  $\G = \{G_{r_1} \subsetneq \cdots \subsetneq G_{s-1}  \subsetneq G_{s+1} \subsetneq \cdots \subsetneq G_{r_2}\}$ have ranks $r_1, \ldots, \widehat{s}, \ldots, r_2$ where $s \neq r_2$. Let $F_1, \ldots, F_m$ be the flats of rank $s$ that are compatible with $\M$. Now let us choose an element $e \notin G_{s+1}$ (and hence $e \notin F_j$ for all $j$), and perform the following computation with $\alpha = \alpha_e$
\begin{eqnarray*}
\alpha \cdot \Sigma_{\M,[r_1,r_2]}(\sigma_{\G}) &=& 
\sum_{\F \supset \G} w(\sigma_{\F}) \alpha(\e_{\F/\G}) - \alpha \left(\sum_{\F \supset \G} w(\sigma_{\F}) \e_{\F/\G} \right) \\
&=& \left(  \sum_{j=1}^m w(\sigma_{\G \cup F_j}) \alpha(\e_{F_j}) - \alpha \left(\sum_{j=1}^m w(\sigma_{\G \cup F_j}) \e_{F_j} \right) \right) \\
&=& 0
\end{eqnarray*}
because none of the $\e_{F_j}$s involve the $e$-th coordinate, so $\alpha$ takes a value of 0 on every term in this sum.

Finally, we compute $\alpha \cdot \Sigma_{\M,[k,k]}$, which is supported on the origin $\bullet$. Summing over the flats $G$ of rank $k$ in $\M$, we compute with $\alpha=\alpha_e$ for any $e \in E$:
\begin{eqnarray*}
\alpha \cdot \Sigma_{\M,[k,k]}(\bullet) &=& \sum_{G} w(\e_G) \alpha(\e_{G/\bullet}) - \alpha \left(\sum_{G} w(\e_G) \e_{G/\bullet} \right) \\
 &=& \sum_{G} |\mu(G)| \alpha(\e_G) - \alpha \left(\sum_{G} |\mu(G)| \e_G \right) \\
 &=& \sum_{G \ni e} |\mu(G)|  \\
 &=& |\mu(\Tr^{r+1-k} \M)| = \mu^k,
\end{eqnarray*}
because the balancing condition for $\Sigma_{\M,[k,k]}$ gives $\sum_{G} |\mu(G)|\e_G = 0$ in $\N_E$, and the last step follows from Weisner's theorem:
\[
\text{ For any lattice $L$ and any element $x \neq \widehat{1}$}, \sum_{x \in L \,: \, x \vee e = \widehat{1}} \mu(x) = 0.
\]
applied to the lattice of flats of $\Tr^{r+1-k} \M$ and the atom $e$.
\end{proof}

\begin{proof} \textbf{ \hspace{-.4cm} 4 of Theorem \ref{thm:main}:}
We use Proposition \ref{prop:truncation} to compute the degree of $\alpha^{r-k}\beta^k \cdot \Sigma_{\M} = \alpha^{r-k}\beta^k \cdot \Sigma_{M[1,r]}$. For $k \neq r$ we have
\[
\deg(\alpha \cdot  \alpha^{r-k-1}\beta^k \cdot \Sigma_{M[1,r]}) = 
\deg(\alpha \cdot \Sigma_{M[k+1,k+1]}) = \mu^k
\]
and for $k \neq 1$ we have
\[
\deg(\beta \cdot \alpha^{r-k}\beta^{k-1} \cdot \Sigma_{M[1,r]}) = 
\deg(\beta \cdot \Sigma_{M[k,k]}) = \mu^k
\]
as desired.  
\end{proof}

This final proof of Theorem \ref{thm:main} is based on \cite{Huhthesis, HuhKatz}.

\section{Further developments}\label{sec:further}

We close with a small selection of recent results that highlight a few additional directions and  techniques in the intersection theory of matroids. They all involve one of the most important functions of a matroid: the \emph{Tutte polynomial}
\[
T_{\M}(x,y) = \sum_{A \subseteq E} (x-1)^{r-r(A)} (y-1)^{|A| - r(A)}.
\]
This is a very powerful invariant, because every matroid invariant that satisfies a deletion-contraction recurrence -- for instance, the characteristic polynomial -- can be expressed in terms of the $T_{\M}(x,y)$. Some of the results also involve the beta invariant $\beta(\M)$, which is the coefficient of $x^1y^0$ and of $x^0y^1$ in $T_{\M}(x,y)$. 

Our goal here is to give a brief description of the key combinatorial aspects of these constructions, but each one of them has an elegant geometric origin. To fully understand the motivation, as well as the relevant definitions, we invite the reader to consult the relevant references.

\medskip

\noindent
\textbf{\textsf{Matroid fans and symmetrized Minkowski weights.}} 
Berget, Spink, and Tseng \cite{BergetSpinkTseng} defined the \emph{one-window symmetrized Minkowski weights} $\Phi_{r,k}$ on the permutahedral variety, and proved that
\[
\deg [{\Sigma_{\M}} \cdot \Phi_{r,k}] =  \textrm{coeff. of $q^k$ in } T_{\M}(1,q) 
\]
for $0 \leq k \leq n-r$.
They computed these degrees by finding the stable intersection of the corresponding Minkowski weights, as described in Section \ref{sec:MW}. They introduced the combinatorial framework of ``sliding sets" to describe the relevant intersection points.

\medskip

\noindent
\textbf{\textsf{Chern-Schwartz-MacPherson cycles of a matroid.}} 
L\'opez de Medrano, Rinc\'on, and Shaw \cite{LopezRinconShaw} defined the \emph{$k$-th Chern-Schwartz-MacPherson (CSM) cycle of a matroid $\M$} 
to be the $k$-skeleton of the matroid fan $\Sigma_{\M}$ with weights 
\[
w(\sigma_{\F}) \coloneqq (-1)^{r-k} \prod_{i=0}^k \beta(\M[F_i, F_{i+1}]), \quad  \F = \{\emptyset = F_0 \subset F_1 \subset \cdots \subset F_k \subset F_{k+1} = E\}
\]
where $\beta(\M[F_i, F_{i+1}])$ is the beta invariant of the minor $\M[F_i, F_{i+1}]$ \cite{LopezRinconShaw}. 
They proved that this is a Minkowski weight on $\Sigma_E$, with degree
\begin{equation} \label{eq:degcsm}
\deg_E (\csm_k(\M) \cdot  \alpha^k) = \text{coefficient of $q^k$ in } \overline{\chi}_{\M}(q+1).
\end{equation}
for $0 \leq k \leq r$. 

They gave a deletion-contraction proof, in the context of Minkowski weights. This required describing the CSM cycles of $\M$ in terms of the CSM cycles of the deletion $\M \backslash e$ and the contraction $\M/e$ for an element $e$ that is not a loop or coloop, using the relevant pushforward and pullback maps. 

Ashraf and Backman \cite{AshrafBackman} gave an alternative proof of \eqref{eq:degcsm} using stable intersections of Minkowski weights, relying on the Gioan-Las Vergnas refined activities expansion of the Tutte polynomial \cite{GioanLasVergnas}.

\medskip

\noindent
\textbf{\textsf{The conormal fan of a matroid.}} Ardila, Denham, and Huh \cite{ArdilaDenhamHuh1, ArdilaDenhamHuh2} introduced the \emph{conormal fan} $\Sigma_{\M,\M^\perp}$ of a matroid $\M$. Its analysis required them to go beyond the permutahedral fan, introducing the \emph{bipermutahedral fan} $\Sigma_{E,E}$.

The conormal fan $\Sigma_{\M,\M^\perp}$ has support 
$|\Sigma_{\M,\M^\perp}| = |\Sigma_{\M}| \times |\Sigma_{\M^\perp}|$
in $\N_E \times \N_E$, where $\M^\perp$ is the dual matroid.
It is a subfan of the bipermutahedral fan $\Sigma_{E,E}$, whose Chow ring contains elements 
$\gamma, \delta \in A^1(\Sigma_{E,E})$, such that
\begin{equation}\label{eq:adh}
\deg_{E,E} [{\Sigma_{\M,\M^\perp}} \cdot \gamma^{k} \delta^{n-1-k}] = (-1)^{r-k}  \textrm{coeff. of $q^k$ in } \overline{\chi}_{\M}(q+1)
\end{equation}
for $0 \leq k \leq r$.

They gave two proofs of \eqref{eq:adh}, based on Brion's presentation of the Chow ring as in Section \ref{sec:A}. One describes the lexicographic expansion of $\gamma^{k} \delta^{n-1-k}$ in the Chow ring of $\Sigma_{\M,\M^\perp}$ in terms of the combinatorics of biflats and biflags \cite{ArdilaDenhamHuh2}, relying on work of LasVergnas on basis activities \cite{LasVergnas}. 
The other one shows that the combinatorially intricate CSM cycles of $\Sigma_{\M}$ are ``shadows" of simpler cycles of $\Sigma_{\M,\M^\perp}$ under the pushforward map of Minkowski weights:
\[
\csm_k(\M) = (-1)^{r-k}\pi_*(\delta^{n-k-1} \cdot 1_{\M,\M^\perp})
\]
for $0 \leq k \leq r$, where $1_{\M,\M^\perp}$ is the top-dimensional constant Minkowski weight on $\Sigma_{\M,\M^\perp}$ \cite{ArdilaDenhamHuh1}.
The projection formula then shows that \eqref{eq:degcsm} implies  \eqref{eq:adh}.

\medskip

\noindent
\textbf{\textsf{Matroid valuations in intersection theory.}}
The \emph{matroid polytope} of a matroid $\M$ is
\[
P_{\M} = \conv\{\e_B \, : \, B \text{ is a basis of } \M\} \subset \RR^E.
\]
A \emph{matroid valuation} is a function $\Phi$ from the set of matroids on $E$ to an additive abelian group such that for any subdivision of a matroid polytope $P_{\M}$  into matroid polytopes $P_{\M_1}, \ldots, P_{\M_k}$  we have the inclusion-exclusion relation
\begin{equation}\label{eq:weakval}
\Phi(\M) = \sum_{i=1}^k (-1)^{\dim P_{\M}  - \dim P_{\M_i}} \Phi(\M_i). 
\end{equation}
This property seems restrictive but 
is surprisingly common \cite{ArdilaFinkRincon, ArdilaSanchez, DerksenFink}. 
For example, $H(\M)=\sum_{\sigma \in S_E}(\sigma, r_{\M}(\{\sigma(1)\}), 
r_{\M}(\{\sigma(1), \sigma(2)\}), \ldots, r_{\M}(\{\sigma(1), \ldots, \sigma(n)\})$ is valuative \cite{ArdilaFinkRincon}, and a very broad range of matroid valuations can be built from it. Notice that $H(\M)$ determines $\M$ entirely.

This framework is relevant and useful in the intersection theory of matroids. For example, the Bergman fan $\Sigma_{\M}$ is a matroid valuation, when regarded as a Minkowski weight on $\Sigma_E$ \cite[Theorem 4.5]{LopezRinconShaw} or as the piecewise polynomial \eqref{eq:PP} on $\Sigma_E$ \cite[Theorem 5.4]{ArdilaFinkRincon}, \cite[Proposition 5.6]{BEST}. 
More generally, the CSM cycles of a matroid are also valuative \cite{LopezRinconShaw}.
The multivariate \emph{volume polynomial} of $\M$, which is equivalent to the Chow ring $A(\M)$, is also valuative \cite{Eur}. 

The case of matroid invariants, which satisfy $f(M_1) = f(M_2)$ when $M_1 \cong M_2$, is best understood. Examples include the characteristic and Tutte polynomials and the beta invariant.
Derksen and Fink described the universal valuative matroid invariant $G(\M)=\sum_{\sigma \in S_E}(r_{\M}(\{\sigma(1)\}), 
r_{\M}(\{\sigma(1), \sigma(2)\}), \ldots, r_{\M}(\{\sigma(1), \ldots, \sigma(n)\})$; this is the symmetrization of $H(\M)$ above. They also showed that any valuative matroid invariant $f$ is determined by its value on \emph{Schubert matroids}.

This gives a powerful way to prove an equation $f(\M)=g(\M)$ for all matroids $\M$:

\noindent
1. Prove that $f$ and $g$ are both valuative (e.g. using the techniques of \cite{ArdilaFinkRincon, ArdilaSanchez, DerksenFink}).

\noindent
2. Prove that $f(\M)=g(\M)$ for all \textbf{Schubert} matroids $\M$ (e.g. by a combinatorial argument that uses the structure of Schubert matroids, or by a geometric argument that works for realizable matroids, which include Schubert matroids).

\medskip

\noindent
\textbf{\textsf{Tautological matroid classes.}} 
Berget, Eur, Spink, and Tseng defined the \emph{tautological Chern classes} $c_i(\mathcal{S}^\vee_{\M}), c_i(\mathcal{Q}_{\M}) \in A^i(\Sigma_E)$ in the permutahedral variety. They gave a valuative proof of the identity
\begin{equation}\label{eq:speyer}
\deg_E [c_r({\mathcal{S}^\vee_{\M}})c_{n-r}({\mathcal{Q}_{\M}}) ] = \beta(\M)
\end{equation}
in the spirit of the previous section. Once one understands the relevant definitions, the valuativity of both sides of \eqref{eq:speyer} follows directly from the discussion there.
They then proved \eqref{eq:speyer} algebro-geometrically for all matroids realizable over $\CC$ -- which includes Schubert matroids -- relying on earlier work of Speyer \cite{Speyer}.

More generally, they showed that the intersections of the tautological Chern classes with powers of the classes $\alpha$ and $\beta$ give the following reparameterization of the Tutte polynomial:
\[
\sum \deg_E [\alpha^i \beta^j c_k({\mathcal{S}^\vee_{\M}})c_l({\mathcal{Q}_{\M}}) ] x^iy^jz^kw^l =  \frac{(y+z)^{r+1}(x+w)^{n-r}}{x+y} T_{\M}\left(\frac{x+y}{y+z}, \frac{x+y}{x+w}\right),
\]
summing over all indices with $i+j+k+l=n$. They used the framework of Section \ref{sec:PP}, analyzing how the piecewise polynomials representing
$\alpha^i \beta^j c_k({\mathcal{S}^\vee_{\M}})c_l({\mathcal{Q}_{\M}})$ behave under deletion-contraction. Our proof in Section \ref{sec:matroidPP} was inspired by theirs.

\medskip

\noindent
\textbf{\textsf{Tropical critical points of affine matroids.}} The affine Bergman fan $\widehat{\Sigma}_{\M}$ is the Bergman fan $\Sigma_{\M}$ with an added lineality space $\RR \e_E$ in $\RR^E$. 
An \emph{affine matroid} $(\M, e)$ consts of a matroid $\M$ and a chosen element $e \in E$. The Bergman fan of $(\M, e)$ is
$\widehat{\Sigma}_{(\M,e)} = \{\x \in \RR^{E-e} \, : \, (0,\x) \in \widehat{\Sigma}_{\M}\}$.
Ardila, Eur, and Penagui\~ao gave two proofs of the following formula conjectured by Sturmfels \cite{Sturmfelsetal}:
\begin{equation} \label{eq:sturmfels}
\deg [{\widehat{\Sigma}_{(M,e)} \cdot ( - \widehat{\Sigma}_{(M/e)^\perp})}] = \beta(\M). 
\end{equation}
Their first proof described the stable intersection of the fans $\widehat{\Sigma}_{(M,e)}$ and $- \widehat{\Sigma}_{(M/e)^\perp}$ explicitly, by developing the framework of \emph{arboreal pairs of set partitions} and connecting to Ziegler's $\beta$nbc bases of ordered matroids \cite{Ziegler}. Their second proof wrote down piecewise polynomials representing the left-hand sides of \eqref{eq:speyer} and \eqref{eq:sturmfels}, and showed that their difference is a multiple of a linear function, and hence equal to $0$ in the Chow ring of $\Sigma_E$. Thus \eqref{eq:speyer} implies \eqref{eq:sturmfels}.

%
%
%
%
%
%
%
%
%

%
%
%
%



\thankyou{I would like to thank the organizers of the Clay Mathematics Institute 
and the 2024 British Combinatorics Conference for the invitation to deliver the Clay Lecture and write this accompanying survey.
I am very thankful to many coauthors and friends with whom I have learned the material in this survey, including 
Carly Klivans,
Chris Eur, 
Dusty Ross, 
Felipe Rinc\'on,
Graham Denham,
Johannes Rau,
June Huh,
Kris Shaw, 
Lauren Williams, 
Mont Cordero--Aguilar, and
Ra\'ul Penagui\~ao.
This work was partially supported by United States National Science Foundation grant  DMS-2154279.}


\bibliographystyle{amsplain}
\bibliography{references}

%
%
%
%
%
%

\myaddress


\end{document}